\documentclass[11pt]{amsart}

\usepackage{amssymb,amsmath,amsthm}
\usepackage[all]{xy}

\newtheorem{thm}{Theorem}[section]
\newtheorem{prop}[thm]{Proposition}
\newtheorem{lem}[thm]{Lemma}

\theoremstyle{definition}
\newtheorem{example}[thm]{Example}
\newtheorem{defn}[thm]{Definition}
\newtheorem{remark}[thm]{Remark}
\newtheorem{hypothesis}[thm]{Hypothesis}

\numberwithin{equation}{section}

\newcommand{\bbZ}{{\mathbb{Z}}}
\newcommand{\bbR}{{\mathbb{R}}}
\newcommand{\bbP}{{\mathbb{P}}}
\newcommand{\bbG}{{\mathbb{G}}}
\newcommand{\bbC}{{\mathbb{C}}}
\newcommand{\GL}{\operatorname{GL}}
\newcommand{\PGL}{\operatorname{PGL}}
\newcommand{\Aut}{\operatorname{Aut}}
\newcommand{\tAut}{\widetilde{\operatorname{Aut}}}
\newcommand{\Hom}{\operatorname{Hom}}
\newcommand{\sHom}{\mathcal{H}\kern -.5pt om}
\newcommand{\Ext}{\operatorname{Ext}}
\newcommand{\End}{\operatorname{End}}
\newcommand{\Spec}{\operatorname{Spec}}
\newcommand{\id}{\operatorname{id}}
\newcommand{\im}{\operatorname{im}}
\newcommand{\Pic}{\operatorname{Pic}}
\newcommand{\Cl}{\operatorname{Cl}}
\newcommand{\calF}{\mathcal{F}}
\newcommand{\calR}{\mathcal{R}}
\newcommand{\calW}{\mathcal{W}}
\newcommand{\calN}{\mathcal{N}}
\newcommand{\calO}{\mathcal{O}}
\newcommand{\cox}{\widetilde} 
\newcommand{\Cox}{\operatorname{Cox}}
\newcommand{\op}{\operatorname{op}}
\newcommand{\Br}{\operatorname{Br}}
\newcommand{\Res}{\operatorname{Res}}
\newcommand{\Div}{\operatorname{div}}
\newcommand{\Nef}{\operatorname{Nef}}
\newcommand{\Split}{{\operatorname{split}}}
\newcommand{\neut}{{\operatorname{neut}}}
\newcommand{\Weil}[1]{\operatorname{R_{#1}}}
\newcommand{\longhookrightarrow}{
	\ensuremath{\lhook\joinrel\relbar\joinrel\rightarrow}}

\begin{document}

\title{Twisted forms of toric varieties}

\date{\today}

\author{Alexander Duncan} 
\thanks{The author was partially supported by
National Science Foundation
RTG grants DMS 0838697 and DMS 0943832.}

\begin{abstract}
We consider the set of forms of a toric variety over an arbitrary field:
those varieties which become isomorphic to a toric variety after base
field extension.
In contrast to most previous work, we also consider arbitrary isomorphisms
rather than just those that respect a torus action.
We define an injective map from the set of forms of a toric variety to a
non-abelian second cohomology set, which generalizes the usual Brauer
class of a Severi-Brauer variety.
Additionally, we define a map from the set of forms of a toric variety
to the set of forms of a separable algebra along similar lines to a
construction of A.~Merkurjev and I.~Panin.
This generalizes both a result of M.~Blunk for del Pezzo surfaces of
degree 6, and the standard bijection between Severi-Brauer varieties and
central simple algebras.
\end{abstract}


\maketitle

\section{Introduction}
\label{sec:intro}

Let $k$ be an arbitrary field.
A \emph{toric variety} is a normal variety $X$ over $k$ with a faithful
action of a torus $T$ which has a dense open orbit.

When the field $k$ is the complex numbers $\bbC$, toric varieties have been
extensively studied (see, for example, \cite{Ful93Introduction} or
\cite{Cox95The-homogeneous}).
Among the simplest examples is the projective line $\bbP^1$ where the torus
$\bbG_m$ acts via
\[
\lambda \cdot (x:y) = (\lambda x : y)
\]
for $\lambda \in \bbC^\times$ with open orbit $x, y \ne 0$.
In most literature,
as part of the definition of ``toric variety''
the torus is identified with the open orbit on which it acts.
This identification is always possible over $\bbC$.

However, when the ground field $k$ is not separably closed,
the torus $T$ may not be the standard split torus $\bbG_m^n$.
Indeed, when $k$ is the real numbers $\bbR$,
the projective line also has an action of the circle
group $S^1$ via multiplication by rotation matrices.
Thus, $\bbP^1_\bbR$ has the structure of a toric variety for the two
non-isomorphic tori $S^1$ and $\bbG_m$.

Moreover, for general base fields the open orbit
may not even have a rational point.
Indeed, when $k = \bbR$, the conic
\[
C = \{ x^2 + y^2 + z^2 = 0 \} \subset \bbP^1(\bbR)
\]
has an action under the circle group $S^1$ via rotation matrices
on the coordinates $x,y$.
Thus the conic $C$ is a toric variety, but has no $\bbR$-points.

When the torus $T$ is split, we call $X$ a \emph{split} toric variety.
Most of the literature on toric varieties is about the split case
since this is the only case when $k=\bbC$.
In \cite{Vos82Projective}, V.~Voskresenski\u\i\ studies toric varieties
for general tori $T$, but assumes that the open orbit has a rational
point; we will call these \emph{neutral} toric varieties.

Recall that a \emph{$k$-form} of a $k$-variety $X$ is a $k$-variety $X'$
such that
\[X \times_{\Spec(k)} \Spec(K) \simeq X' \times_{\Spec(k)} \Spec(K)\]
for some field extension $K/k$.
For example, $\bbP^1_\bbR$ and the conic $C$ both become isomorphic to
$\bbP^1_\bbC$ after base extension to $\bbC$.
All toric varieties are $k$-forms of a split toric variety.
The main goal of this paper is to study the set of isomorphism classes
of forms of toric varieties.

Note that there are several different natural notions of isomorphism
that one can use.
We consider three different categories of toric varieties
(made precise in Section~\ref{sec:prelimToric}),
each of which has a different notion of isomorphism, and thus, of
$k$-form.

When the toric variety is \emph{neutral} we can identify the
torus with its open orbit and consider the category $\calW$ with
\emph{toric morphisms} which restrict to group homomorphisms of the
tori;
this is the notion of isomorphism studied in \cite{Vos82Projective}.
The category $\calN$ consists of toric varieties with torus-equivariant
morphisms;
this category is implicit in, for example, \cite{VosKly84Toric},
\cite{MerPan97K-theory}, and \cite{EliLimSot14Arithmetic}.
Finally, the category $\calR$ consists of toric varieties with arbitrary
morphisms which completely ignore the toric structure.
To see how the different categories differ, 
note that the two different torus actions on $\bbP^1_\bbR$ discussed
above give rise to two distinct isomorphism classes in $\calN$, but
they are not distinct in $\calR$ since the underlying varieties are the
same.

Our approach is a direct generalization of standard techniques for
studying forms of projective space.  We review this now.
Recall that a \emph{Severi-Brauer variety} is a $k$-form of projective space.
The set of isomorphism classes of $k$-forms of a (suitably nice) algebraic
object is in bijection with the \'etale cohomology set $H^1(k,\Aut(X))$,
which is simply Galois cohomology since $k$ is a field.
Thus the set of isomorphism classes of Severi-Brauer varieties
is in bijection with the set $H^1(k,\PGL_n)$.

Recall that the projective space $X=\bbP^{n-1}$ has a homogeneous
coordinate ring whose spectrum is a vector space $V \simeq k^n$.
The space $X$ is a quotient of an open subset of the vector space $V$
by the group $\bbG_m$.
This construction gives rise to an exact sequence
\begin{equation} \label{eq:introPnSeq}
1 \to \bbG_m \to \GL_{n} \to \PGL_{n} \to 1
\end{equation}
where $\GL_n$ acts on $V$ and $\PGL_{n}$ is the automorphism group of
$\bbP^{n-1}$.

From the exact sequence \eqref{eq:introPnSeq}, the long exact sequence in
non-abelian Galois cohomology produces a well-known injection
\begin{equation} \label{eq:introP2H2inj}
H^1(k, \PGL_{n} ) \hookrightarrow H^2(k, \bbG_m )
\end{equation}
where $H^1(k, \PGL_{n} )$ is in bijection with the set of
isomorphism classes of Severi-Brauer varieties of dimension $n-1$
and the group $H^2(k, \bbG_m )$ is the Brauer group $\Br(k)$.

For a complete split toric variety $X$, the \emph{Cox ring} of $X$,
introduced in \cite{Cox95The-homogeneous},
generalizes the usual homogeneous coordinate ring for projective space.
Again, the spectrum of the Cox ring is a vector space $V$ that has an
open subset from which $X$ can be reconstructed as a quotient by a
diagonalizable group $S$.
Here $S$ is the Cartier dual of the class group of $X$,
which generalizes $\bbG_m$ from the case where $X$ is $\bbP^{n-1}$.

There exists a linear algebraic group $\tAut(X)$ acting on the Cox ring,
generalizing $\GL_n$ from the case of projective space,
which has a more convenient description than
the automorphism group scheme $\Aut(X)$.
Under certain assumptions on $X$,
the standard sequence \eqref{eq:introPnSeq} is generalized by the bottom
row of the following commutative diagram with exact rows
\begin{equation} \label{eq:twoLines}
\xymatrix{
1 \ar@{->}[r] &
S \ar@{->}[r] \ar@{=}[d] &
\cox{T} \rtimes W \ar@{->}[r] \ar@{->}[d] &
T \rtimes W \ar@{->}[r] \ar@{->}[d] &
1 \\
1 \ar@{->}[r] &
S \ar@{->}[r] &
\tAut(X) \ar@{->}[r] &
\Aut(X) \ar@{->}[r] &
1 \\
}
\end{equation}
where $T$ is a maximal torus of $\Aut(X)$,
the group $\cox{T}$ is a maximal torus of $\tAut(X)$,
and $W$ is the group of toric automorphisms of $X$.
The diagram \eqref{eq:twoLines} is essentially due to D.~Cox;
it follows from Theorem~\ref{thm:twoLines} below.

Galois cohomology can be viewed as a functor $H^1(k,-)$ which behaves
well in exact sequences.
In Theorem~\ref{thm:mainSquare}, we show how the sets of isomorphism
classes for the three categories, as well as their relationships to
each other and to their subsets of neutral forms, can be readily
obtained from the Galois cohomology sets associated to
\eqref{eq:twoLines}.
Applying $H^1(k,-)$ to the rightmost square we may intrepret each set as
the $k$-forms of $X$ (up to isomorphism) in an appropriate category:
\begin{equation} \label{eq:mainSquareInterp}
\xymatrix{
{\text{\parbox{3cm}{
\centering
\begin{tabular}{|c|}
\hline
Forms in $\calW$\\
\hline
neutral $X$\\
specified $T$\\
\hline
\end{tabular}
}}}
\ar@{^{(}->}[r] \ar@{->>}[d] &
{\text{\parbox{3cm}{
\centering
\begin{tabular}{|c|}
\hline
Forms in $\calN$\\
\hline
arbitrary $X$\\
specified $T$\\
\hline
\end{tabular}
}}}
\ar@{->>}[d] \ar@{-->}@/_1pc/[l] \\
{\text{\parbox{3cm}{
\centering
\begin{tabular}{|c|}
\hline
 - \\
\hline
neutral $X$\\
unspecified $T$\\
\hline
\end{tabular}
}}}
\ar@{^{(}->}[r] \ar@{-->}@/_1pc/[u] &
{\text{\parbox{3cm}{
\centering
\begin{tabular}{|c|}
\hline
Forms in $\calR$\\
\hline
arbitrary $X$\\
unspecified $T$\\
\hline
\end{tabular}
}}}
\ar@{-->}@/_1pc/[l] \\
}
\end{equation}
Here the vertical maps are surjections, the horizontal maps are
injections and the dashed arrows are canonical sections or retracts.
We see that the isomorphism classes are naturally partitioned into
\emph{neutralization classes} each of which contains exactly one neutral
toric variety.

Note that $\Aut(X)$ has a natural induced action on $\Cl(X)$ and thus on
its Cartier dual $S$.
If $J$ is the image of the action of $\Aut(X)$ on $\Cl(X)$,
then the group $J$ is finite but non-trivial in general.
Thus, unlike the sequence \eqref{eq:introPnSeq} for projective space,
$S$ is not necessarily central in \eqref{eq:twoLines}
and one should not expect a map $H^1(k,\Aut(X)) \to H^2(k,S)$
for a general toric variety.
Nevertheless, using the theory of non-abelian $H^2$
from~\cite{Spr66Nonabelian}~and~\cite{Gir71Cohomologie},
we prove in Theorem~\ref{thm:H2main} that there is an injection
\begin{equation} \label{eq:introH2inj}
H^1(k, \Aut(X) ) \hookrightarrow H^2(k, S \to J )
\end{equation}
where $H^2(k, S \to J )$ is a structured set defined below 
which is a natural analog of the Brauer group.
One may view the map \eqref{eq:introH2inj} as a refinement
of the \emph{elementary obstruction} from \cite{ColSan87La-descente}.

Finally, given a split smooth projective toric variety $X$, we construct
a canonical $k$-algebra $B$ in a similar vein as a construction of
A.~Merkurjev and I.~Panin in \cite{MerPan97K-theory}.
To each $k$-form of $X$ we associate a $k$-form of $B$.
This construction is a common generalization of the usual association of
a Severi-Brauer variety to a central simple algebra and a construction
of M.~Blunk~\cite{Blu10Del-Pezzo} for del Pezzo surfaces of degree $6$.
We investigate when the isomorphism class of the $k$-forms of $X$ can be
recovered from the isomorphism classes of $B$; it turns out to be
closely related to retract rationality of the $k$-forms of $X$.

The paper is structured as follows.
In Sections \ref{sec:prelims}, \ref{sec:prelimToric}, and
\ref{sec:splitToric}, we fix notation, state definitions, and review
basic facts that will be needed later in the paper.
In Section \ref{sec:twists}, we use Galois cohomology and the
structure theory of split toric varieties to classify their forms.
In Section \ref{sec:twistedProps}, we extend the structure theory
of the split toric variety to the general case.
In Section \ref{sec:non-abelianH2}, we define the set $H^2(k,S \to J)$
and prove the injection \eqref{eq:introH2inj}.
In Sections~\ref{sec:H2approx},~\ref{sec:MerkPanin},
and~\ref{sec:approximation}, we use the preceeding theory to
investigate when Blunk's construction can be generalized to
other toric varieties.

\section{Preliminaries}
\label{sec:prelims}

Let $k$ be a field.  We will denote by $k_s$ a separable closure
of $k$.  We denote by $\Gamma_k$ the absolute Galois group of the field
$k$, which is a profinite group.

A \emph{variety} $X$ is a geometrically integral separated scheme of
finite type over a field $k$.
A \emph{group scheme} $G$ will always be a group scheme of finite type
over a field $k$.
An \emph{algebraic group} $G$ is a smooth group scheme of finite
type over a field $k$.

Given a field extension $K/k$ we denote by
\[ X_K := X \times_{\Spec(k)} \Spec(K) \]
the pullback, which is a variety defined over $K$.
For the separable closure,
we use the shorthand $\overline{X} := X_{k_s}$.
A \emph{$k$-form of $X$}, is a variety $X'$ defined over $k$
such that $X'_K \simeq X_K$ for some field extension $K/k$.

We assume that reader is familiar with Galois cohomology
(see, e.g., \cite{Ser02Galois}).
For an algebraic group $G$, $H^i(k,G)$ will denote the
$i$th Galois cohomology set $H^i(\Gamma_k,G(k_s))$.
This is an abelian group when $G$ is
abelian, a group when $i=0$, and a pointed set
when $G$ is non-abelian and $i=1$.

Of fundamental importance to this paper is the well-known
functorial bijection of pointed sets
\[
H^1(k,\Aut(X)) \simeq
\left\{
{\text{\parbox{4cm}{
\centering
\begin{tabular}{c}
isomorphism classes\\
of $k$-forms of $X$
\end{tabular}
}}}
\right\}\]
which holds when $X$ is quasiprojective and $\Aut(X)$ is an
algebraic group.

For an algebraic group $G$,
an element $\gamma \in H^1(k,G)$ can be represented by a
\emph{cocycle} $c$ or by a $G$-torsor $T$ (also called a principal
homogeneous space).
If $X$ is an algebraic variety (or algebraic group, or $k$-algebra)
with a $G$-action, then we can construct the twisted variety ${}_cX$
or ${}^TX$ as in \S 5.3 of \cite{Ser02Galois}.

\subsection{Algebras}

We refer the reader to \S 1, \S 18, and \S 23 of \cite{KnuMerRos98The-book}
for many of the results that follow.
We assume throughout that all algebras are associative and unital.
Given an algebra $A$, we denote its opposite algebra by $A^{\op}$.

An \emph{\'etale $k$-algebra $E$} is a direct product
\[ E = F_1 \times \cdots \times F_r \]
where $F_1,\ldots, F_r$ are separable field extensions of $k$.
An \'etale algebra $E$ is \emph{split} if every field $F_i$ in the
decomposition is isomorphic to $k$.
The \emph{degree} of an \'etale algebra $E$ is its dimension as a
vector space over $k$.

A \emph{central simple algebra over $k$} is a $k$-algebra $A$ such that
there exists a field $K$ for which $A_K \simeq M_n(K)$ where $M_n(K)$ is
the algebra of $n \times n$-matrices over $K$.  The algebra $A$ is
\emph{split} if $A \simeq M_n(k)$ over the original field.

A \emph{separable algebra $A$} is a finite-dimensional $k$-algebra
which is a finite product
\[ A = A_1 \times \cdots \times A_r \]
where each $A_i$ is a central simple algebra over a finite separable
field extension $F_i$ of $k$.
The algebra $A$ is \emph{neutral} if every algebra $A_i$ is
split as a central simple algebra over $F_i$.
The algebra $A$ is \emph{split} if $F_1 = \cdots = F_r = k$
and every central simple algebra $A_i$ is split over $k$.
Note that the number of simple algebras $r$ may increase after a base
field extension.

Every separable algebra has a \emph{split form} $A_\Split$
which is the unique split $k$-form of $A$.
Every separable algebra $A$ has a \emph{neutralization} $A_\neut$ where
we replace each $A_i$ in the product with its split form as an
$F_i$-algebra.
We define a \emph{neutralization class} of a separable $k$-algebra $A$ to be
the set of $k$-forms of $A$ which have the same neutralization.
The center $Z(A)$ of a separable $k$-algebra is an \'etale $k$-algebra.
The center $Z(A)$ is invariant within neutralization classes.

\subsection{Automorphisms}

Given a $k$-algebra $A$, let $\Aut(A)$ denote the group scheme of
automorphisms of $A$.

If $E$ is an \'etale $k$-algebra of degree $n$, then
$\Aut(E)$ is a form of the symmetric group $S_n$;
thus, $\Aut(E)$ is finite \'etale over $k$.

Given a finite dimensional $k$-algebra $A$, we denote by
$\GL_1(A)$ the algebraic group representing the functor
\[ R \to (A_R)^\times \]
on commutative $k$-algebras $R$ (see \S 20 of~\cite{KnuMerRos98The-book}).

Let $A$ be a separable $k$-algebra with center $Z(A)$.
By \S 23 of \cite{KnuMerRos98The-book}, the connected
component of $\Aut(A)$ is given by
\[ \Aut(A)^\circ \simeq \GL_1(A)/\GL_1(Z(A)) \]
and $\pi_0(\Aut(A))$, the quotient by the connected component, is a
subgroup of $\Aut(Z(A))$.

Given an algebraic subgroup $I$ of $\pi_0(\Aut(A))$, the \emph{$I$-restricted
automorphism group of $A$}, denoted $\Aut_I(A)$, is the preimage of $I$
in $\Aut(A)$.
Note, in particular, when $I$ is trivial, $\Aut_1(A) = \Aut(A)^\circ$.

\subsection{Groups of multiplicative type}

Most of the material here can be found in,~e.g.,~\cite{Vos98Algebraic}.

Let $\bbG_m = \GL_1(k)$.
A group scheme $S$ of finite type is \emph{diagonalizable} if $S$ is a
closed subgroup of $(\bbG_m)^n$ for some positive integer $n$.
A \emph{group of multiplicative type} is a group scheme $S$
such that $\overline{S}$ is diagonalizable.
A group of multiplicative type is \emph{split}
if it is diagonalizable.
An algebraic group $S$ is a \emph{torus} if
$\overline{S} \simeq (\bbG_m)^n$  for some non-negative integer $n$.

Let $\Gamma$ be a profinite group.
A \emph{$\Gamma$-module} $L$ is a finitely generated abelian group $L$
with a continuous action of $\Gamma$ where $L$ is endowed with the
discrete topology.
A \emph{$\Gamma$-lattice} is a torsion-free $\Gamma$-module.

There is an exact anti-equivalence between the
category of groups of multiplicative type
and the category of $\Gamma_k$-modules
which we will call ``duality.''
Given a group $S$ of multiplicative type, the \emph{character group},
$\widehat{S}$, is the corresponding $\Gamma_k$-module.
Conversely, given a $\Gamma_k$-module $L$, the corresponding group of
multiplicative type will be denoted $D(L)$.
Under this equivalence, tori correspond to $\Gamma$-lattices.

We remark that if the torsion subgroup of the underlying abelian group of
a $\Gamma_k$-module $L$ has order relatively prime to the
characteristic, then $D(L)$ is an algebraic group (in other words, it is
smooth).

The image of the map $\Gamma_k \to \Aut(\widehat{S})$ is a finite group
which we call the \emph{decomposition group of $S$}.
One can also define the \emph{cocharacter group of $S$} as
the $\Gamma_k$-lattice $\Hom(\widehat{S}, \bbZ )$.
One can recover the original group scheme $S$ from its cocharacter lattice if
and only if $S$ is a torus.

\subsection{Weil Restrictions and Galois Cohomology}

Let $E$ be an \'etale $k$-algebra and let $\calF$ be a functor from
$E$-algebras to sets.
We define the \emph{Weil restriction} $\Weil{E/k}\calF$
of $\calF$ as the functor from commutative $k$-algebras to sets given by
\[
\Weil{E/k}\calF(R) = \calF(R \otimes_k E)
\]
for each $k$-algebra $R$.
The Weil restriction of an algebraic group (resp. variety) is also
an algebraic group (resp. variety).

\begin{lem} \label{lem:Shapiro}
Let $E$ be an \'etale $k$-algebra and $G$ be an algebraic group over
$E$.
There is a natural isomorphism
$H^i(k,\Weil{E/k}G) \simeq H^i(E,G)$
of groups (or pointed sets).
\end{lem}

\begin{proof}
See Lemma~29.6 of~\cite{KnuMerRos98The-book}.
\end{proof}

\begin{prop}[Hilbert 90] \label{prop:Hilbert90}
For any separable $k$-algebra $A$, the cohomology set
$H^1(k,\GL_1(A))$ is trivial.
\end{prop}

\begin{proof}
See Theorem~29.2 of \cite{KnuMerRos98The-book}.
\end{proof}

\begin{prop} \label{prop:cGLisGLc}
For any separable $k$-algebra $A$ and for any cocycle
$c$ representing an element in $H^1(k,\Aut(A))$,
we have
\[ {}_c\GL_1(A) = \GL_1({}_cA) \]
where ${}_cA$ (resp. ${}_c\GL_1(A)$) denotes the form
of $A$ (resp. $\GL_1(A)$) twisted by the cocycle $c$.
\end{prop}

\begin{proof}
The embedding $\GL_1(A)(k_s) \to A_{k_s}$
is $\Aut(A_{k_s})(k_s)$-equivariant. 
\end{proof}

\subsection{Permutation lattices and quasi-split tori}

A $\Gamma$-lattice $L$ is \emph{permutation} if $L$ has a basis which is
permuted by $\Gamma$.  A torus $S$ is \emph{quasi-split} if
$\widehat{S}$ is permutation.

There is an antiequivalence of categories
between the category of \'etale algebras and the category of
$\Gamma_k$-sets
where $\Gamma_k$-orbits correspond to the subfields $F_i$ of $E$.

Given an \'etale $k$-algebra $E$, the group $\GL_1(E)$ is a torus.
Indeed, \[ T = \GL_1(E) \simeq \Weil{E/k}\bbG_{m,E} \]
is a Weil restriction.
The character lattice $\widehat{T}$ is a permutation $\Gamma_k$-lattice
with a basis indexed by the $\Gamma_k$-set corresponding to $E$.
The torus $T$ is always a quasi-split torus and any quasi-split torus can
arise in this way.
Note, however, that non-isomorphic \'etale algebras may give rise to
isomorphic tori since the choice of basis is not canonical.

\section{Preliminaries on General Toric Varieties}
\label{sec:prelimToric}

There are several reasonable notions of a ``toric
variety'' over a general field.
Here we fix the definitions for the remainder of the paper.

\begin{defn} \label{def:toricWithT}
Let $T$ be a torus.  A \emph{toric $T$-variety} $X$ is a normal $k$-variety
with a faithful $T$-action
and a dense open $T$-orbit $X_0$.
A toric $T$-variety $X$ is \emph{neutral} if there exists a $T$-equivariant
isomorphism $T \to X_0$.
A toric $T$-variety $X$ is \emph{split} if $T$ is a split torus.
\end{defn}

Over an algebraically closed field, the notions of neutral and
split are vacuous.
It is often desirable to keep track of a specific
isomorphism $T \to X_0$ when $X$ is neutral, but this is not part of our
definition.
Note that there are no non-trivial torsors under a split torus, so a
split toric $T$-variety is always neutral.

\begin{defn} \label{def:toricWithoutT}
We say $X$ is a \emph{toric variety} if there exists a torus $T$ with an
action on $X$ giving $X$ the structure of a toric $T$-variety.  We say
$X$ is \emph{neutral} if one can choose $T$ such that $X$ is neutral as
a toric $T$-variety.  We say $X$ is \emph{split} if one can choose $T$
such that $X$ is split as a toric $T$-variety.
\end{defn}

The difference between Definitions \ref{def:toricWithT} and
\ref{def:toricWithoutT} is whether a specific torus is
fixed a priori or not.
Note that a toric variety $X$ is neutral for any choice of torus if it
is neutral for any one choice.
However, beware that a split toric variety may also be a toric $T$-variety
for a different torus $T$ that is \emph{not} split.

We introduce 3 different categories of toric varieties to emphasize the
different kind of morphisms one might consider in light of the above
considerations.  In Section~\ref{sec:twists}, these categories 
will be used as natural settings for the machinery of descent.

\begin{defn}
The category $\calR$:
\begin{enumerate}
\item objects are toric varieties,
\item morphisms are morphisms of varieties.
\end{enumerate}
\end{defn}

\begin{defn}
The category $\calN$:
\begin{enumerate}
\item objects are pairs $(T,X)$
where $T$ is a torus and $X$ is a toric $T$-variety,
\item morphisms from $(T,X)$ to $(T',X')$ are pairs $(g,f)$
where $g : T \to T'$ is a group homomorphism
and $f : X \to X'$ is a morphism of varieties
which is $T$-equivariant via $g$.
\end{enumerate}
\end{defn}

The automorphisms in the category $\calN$ amount to
automorphisms of the subvariety $X_0$ which extend to all of $X$.

\begin{defn}
The category $\calW$:
\begin{enumerate}
\item objects are triples $(T,X,\iota)$
where $T$ is a torus, $X$ is a neutral toric $T$-variety,
and $\iota : T \hookrightarrow X$ is an isomorphism with the dense open orbit,
\item morphisms from $(T,X,\iota)$ to $(T',X',\iota')$ are pairs $(g,f)$
where $g : T \to T'$ is a group homomorphism
and $f : X \to X'$ is a morphism of varieties such that the diagram
\[
\xymatrix{
X \ar[r]^f &
X' \\
T \ar[r]^g \ar@{^{(}->}[u]^{\iota} &
T' \ar@{^{(}->}[u]^{\iota'}
}
\]
commutes.
\end{enumerate}
\end{defn}

The morphisms of $\calW$ are called \emph{toric morphisms}
in the literature.
Observe that multiplication by a non-trivial element of the torus is
an automorphism in $\calN$, but not in $\calW$.

We mention two useful results regarding general toric varieties.

\begin{prop} \label{prop:compactification}
If $T$ is a torus then there exists a smooth projective $T$-variety $X$.
\end{prop}

\begin{proof}
See \cite{ColHarSko05Compactification}.
\end{proof}

\begin{prop} \label{prop:rationalPoint}
A smooth projective toric variety $X$ is neutral if and only if $X$ has a
rational $k$-point.
\end{prop}

\begin{proof}
See Proposition~4~of~\cite{VosKly84Toric}.
\end{proof}

\section{Structure of Split Toric Varieties}
\label{sec:splitToric}

Throughout this section, $X$ is a proper split toric $T$-variety with a
specified embedding $T \hookrightarrow X$.

\begin{hypothesis} \label{assume}
Throughout the paper we will make the following additional technical
assumptions:
\begin{itemize}
\item the automorphisms functor $\Aut(X)$ is a linear algebraic group,
and
\item the order of the torsion subgroup of the class group $\Cl(X)$
is relatively prime to the characteristic of $k$.
\end{itemize}
We shall see that both these assumptions hold when $X$ is smooth
or $X$ has characteristic $0$.
\end{hypothesis}

Here we outline the well-known structure theory of a split toric variety
to fix notation.
We assume the reader is familiar with standard references on toric
varieties (for example, \cite{Ful93Introduction} or \cite{CoxLitSch11Toric}).
Many references only consider the base field $\bbC$, but much
of the theory goes through unchanged in the split case.

Let $M=\widehat{T}$ be the character lattice of $T$,
and let $N$ be the cocharacter lattice of $T$.
Of course, since $T$ is split, both lattices have trivial
$\Gamma_k$-action.

Split toric $T$-varieties $X$ are in bijective correspondence with
\emph{fans} in their cocharacter lattices $N$.  From the data of a fan
$\Sigma$,
one can determine whether $X$ is smooth, projective, or proper over $k$.
We denote by $\Sigma(k)$ the set of cones of dimension $k$ in $\Sigma$;
in particular, $\Sigma(1)$ is the set of rays.
We denote by $\cox{M}$ the free abelian group with basis indexed by
$\Sigma(1)$.  The dual lattice $\cox{N}$ is canonically isomorphic to
$\cox{M}$ by using this basis.
The lattice $\cox{M}$ is isomorphic to
the group of $T$-invariant Weil divisors of $X$.

For a split toric variety, we have canonical isomorphisms
$\Cl(X) = \Cl(\overline{X})$ and $\Pic(X) = \Pic(\overline{X})$
(this follows from their descriptions via torus-invariant divisors as in,
e.g., Theorem~4.1.3~of~\cite{CoxLitSch11Toric}).
The divisor class group, $\Cl(\overline{X})$, has a natural
structure as a (trivial) $\Gamma_k$-module.
Let $S$ be the group of multiplicative type $D(\Cl(\overline{X}))$.
Note that by our assumption above, $S$ is a linear algebraic group.

There is an exact sequence
\begin{equation} \label{eq:TTS}
1 \to M \to \cox{M} \to \Cl(\overline{X}) \to 1
\end{equation}
when $k_s[\overline{X}]^\times \simeq k_s$.
When $X$ is smooth, $\Cl(\overline{X})$ is canonically isomorphic to
$\Pic(\overline{X})$ and is torsion free.

\subsection{Cox rings}

We now review the theory of Cox rings (see \cite{Cox95The-homogeneous}).
For non-split toric varieties this can be subtle
(see Section \ref{sec:twistedProps} and \cite{DerPie14Cox-rings}),
but here we assume $X$ is split and proper.

The \emph{Cox ring} $\Cox(X)$ of $X$ is the polynomial ring
\[ \Cox(X) := k[x_1, \ldots, x_r] \]
where the monomials $x_1, \ldots, x_r$ correspond to the rays
$\rho_1, \ldots, \rho_r$ in $\Sigma(1)$.
Note that $\Cox(X)$ has a canonical embedding into
$k[\cox{T}]$ where monomials can be identified with elements of $\cox{M}$.
The ring $\Cox(X)$ has a natural $\Cl(\overline{X})$ grading via the morphism
$\cox{M} \to \Cl(\overline{X})$ from \eqref{eq:TTS}.

For a Weil divisor $D$, we denote the graded component of $\Cox(X)$
corresponding to $[D] \in \Cl(\overline{X})$ by $\Cox(X)_{[D]}$
or $\Cox(X)_D$.
Denoting $\calO_X(D)$ as the reflexive sheaf associated to $D$,
there are isomorphisms
\[ \Cox(X)_{D} \simeq H^0(X,\calO_X(D))
= \{ f \in k(X)^\times \colon \Div(f) + D \ge 0 \} \cup \{0\} \]
for every Weil divisor $D$.

We define the \emph{irrelevant ideal} $B$ of $\Cox(X)$ as the monomial
ideal generated by products $x_{i_1}\cdots x_{i_r}$ corresponding to
subsets of rays $\{ \rho_{i_1}, \ldots, \rho_{i_r} \}$ which are the
complement of a cone in $\Sigma$.
Note that $V = \Spec(\Cox(X))$ is an affine variety with a natural
vector space structure.
The ideal $B$ cuts out a closed subvariety $Z \subset V$ whose
complement $\cox{X}$ we call the \emph{characteristic space} of $X$.

Since $S$ is the Cartier dual of $\Cl(\overline{X})$,
the $\Cl(\overline{X})$-grading on $\Cox(X)$ corresponds to
generically-free actions of $S$ on $V$ and on $\cox{X}$
(obtained via restrictions of the action of the torus $\cox{T}$).
We may recover $X$ as the categorical quotient of $\cox{X}$ by $S$;
over $\bbC$ this is Theorem~2.1~of~\cite{Cox95The-homogeneous},
but the proof works in general with minor modifications.
In the case that $X$ is smooth, $S$ is a torus and, thus, the
action of $S$ on $\cox{X}$ is free and the quotient
\[ \psi : \cox{X} \to X \]
is an $S$-torsor (in fact, a \emph{universal torsor} in the sense of
\cite{ColSan87La-descente}).

\subsection{Automorphisms}

The automorphism group scheme of a smooth proper split toric variety was
determined in~\cite{Dem70Sous-groupes}.
Our exposition is heavily inspired by
\cite{Cox95The-homogeneous} where the automorphism group of a split
proper simplicial toric variety over $\bbC$ is determined indirectly via the Cox ring
(see also~\cite{Cox14Erratum}).
The simplicial hypothesis is removed in \cite{Buh96Homogener}.
The automorphism group of a projective split toric variety is
determined in \cite{BruGub99Polytopal} over an arbitrary field.

Recall that we assume that $\Aut(X)$ is a linear algebraic group.
The author knows no counterexample to this assumption for an arbitrary
proper toric variety and it holds many cases of interest:

\begin{lem} \label{lem:AutXsmooth}
Let $X$ be a proper split toric variety and
suppose that either
\begin{enumerate}
\item $X$ is smooth, or
\item $k$ has characteristic $0$ and $X$ is projective.
\end{enumerate}
Then the functor $\Aut(X)$ of automorphisms of $X$ is a linear
algebraic group.
\end{lem}

\begin{proof}
Note that since $X$ is proper, $\Aut(X)$ is a group scheme of locally
finite type over $k$ (see
Theorem~3.7~of~\cite{MatOor67Representability}).
When $X$ is smooth, $\Aut(X)$ is a linear algebraic group
by Proposition~11~of~\cite{Dem70Sous-groupes}.
In characteristic $0$, the neutral component $\Aut(X)^\circ$
must be smooth.
From \cite{BruGub99Polytopal}, we have a description of $\Aut(X)(K)$ for
any field $K$ when $X$ is projective.
From this description we see that $\Aut(X)$ has finitely
many connected components and thus $\Aut(X)$ is a linear algebraic
group.
\end{proof}

Let $W$ be the group of toric automorphisms of $X$
(the subgroup of $\GL(N) \simeq \GL_n(\bbZ)$ which takes cones to cones).
Note that $W$ has induced actions on $\cox{T}$ and $T$.

Let $V_\lambda$ be the weight subspace of $V$ corresponding to $\lambda$
in $\Cl(\overline{X})$ and let $n_\lambda$ be its dimension.
Let $\Lambda$ be the subset of $\Cl(\overline{X})$ corresponding to the
non-trivial weight subspaces $V_\lambda$ of $V$.

\begin{defn} \label{def:CoxAlg}
The \emph{Cox endomorphism algebra of $X$} is the split separable $k$-algebra
\[ A := \prod_{\lambda \in \Lambda} \End(V_\lambda) \ . \]
\end{defn}

We define $W^\circ = W \cap \GL_1(A)$ and note that
\[ W^\circ \simeq \prod_{\lambda \in \Lambda} S_{n_\lambda} \]
where each $S_{n_\lambda}$ is the symmetric group on $n_\lambda$
letters.
The group $W^\circ$ is isomorphic to the Weyl group of $\GL_1(A)$.

\begin{defn} \label{def:J}
The \emph{group of class group automorphisms of $X$}, denoted $J$,
is the image of the map $\Aut(X) \to \Aut(S)$
(recall that we obtain an induced action of $\Aut(X)$ on $S$ since it is
the dual of the class group $\Cl(X)$).
The group $J$ is a finite constant group which will be of fundamental
importance for the remainder of the paper.
\end{defn}

Let $\tAut(X)$ be the normalizer of $S$ in the automorphism group
functor of $\cox{X}$.  We will see that, in fact, $\tAut(X)$ is a linear
algebraic group.

\begin{thm} \label{thm:twoLines}
Let $X$ be a projective split toric variety
with $\Aut(X)$ smooth.
Diagram \eqref{eq:twoLines}:
\[
\xymatrix{
1 \ar@{->}[r] &
S \ar@{->}[r] \ar@{=}[d] &
\cox{T} \rtimes W \ar@{->}[r] \ar@{->}[d] &
T \rtimes W \ar@{->}[r] \ar@{->}[d] &
1 \\
1 \ar@{->}[r] &
S \ar@{->}[r] &
\tAut(X) \ar@{->}[r] &
\Aut(X) \ar@{->}[r] &
1 \\
}
\]
commutes and has exact rows.
Moreover,
$T \rtimes W$ (resp. $\cox{T} \rtimes W$)
is the normalizer of a maximal torus in $\Aut(X)$
(resp. $\tAut(X)$).
There is a unipotent subgroup $U$ such that
\[ \tAut(X) \simeq
U \rtimes \GL_1(A) \rtimes J , \]
and an isomorphism
\[
W \simeq W^\circ \rtimes J
\]
where the splitting is unique up to conjugacy.
\end{thm}

\begin{proof}
When $k = \bbC$ and $X$ is simplicial, the commutative diagram
\eqref{eq:twoLines} is essentially the main theorem of \S 4
of~\cite{Cox95The-homogeneous}.
From \cite{BruGub99Polytopal},
one can obtains the commutativity and exactness of the diagram
on the level of $k$-points where $k$ is an arbitrary field.
The group $\tAut(X)$ can be explicitly constructed as a closed subgroup
of a general linear group acting on a subspace of $k[V]$
(the exposition in \cite{Cox14Erratum} is especially clear).
Using this description, one checks that its Lie algebra has the
appropriate dimension and thus it is a linear algebraic group.

Since $\Aut(X)$ is also a linear algebraic group, we may establish the
remaining statements by frst assuming $k$ is algebraically closed and
then showing that all of the structural maps are actually defined over
the original field $k$.
We need to establish the splittings of $\tAut(X)$ and $W$.
By the references above, the group $\tAut(X) \cap \GL(V)$ is
isomorphic to the product $\GL_1(A)W$.
Moreover, there is a surjective homomorphism $\tAut(X) \to \GL_1(A)W$
with kernel a unipotent subgroup $U$.

Note that $V$ is the vector space dual to the subspace spanned by the
generators $x_1, \ldots, x_r$ of $\Cox(X)$ and thus is spanned by the
dual basis $x_1^*, \ldots, x_r^*$ which may be identified with rays of
$\Sigma$.  The group $W$ permutes these basis vectors.  The group
$W^\circ$ consists of all permutations of $x_1^*, \ldots, x_r^*$ which
preserve the decomposition into weight subspaces $V_\lambda$.

Since $\GL_1(A)$ is connected and $J$ is a constant finite group, we
obtain an exact sequence
\[ 1 \to W^\circ \to W \to J \to 1 \]
which we want to show is split.  The group $J$ acts faithfully on the
set $\Lambda$.
Note that if $W$ splits as $W^\circ \rtimes J$,
then $\GL_1(A)W$ splits as $\GL_1(A) \rtimes J$.

Choose orderings for $x_1^*, \ldots, x_r^*$ within each subspace
$V_\lambda$.  Pick a set-theoretic section $s : J \to W$.
For each $j \in J$, we have $s(j)(V_\lambda) = V_{j(\lambda)}$.
By comparing the ordered bases of $V_\lambda$ and $V_{j(\lambda)}$,
the element $j$ gives rise to an element $w_\lambda \in S_{n_\lambda}$.
Taking the product of the $w_\lambda$'s for each $\lambda \in \Lambda$,
we obtain an element $w_j \in W^\circ$ such that
$w_j^{-1} \circ s(j) \colon V_\lambda \to V_{j(\lambda)}$
is an isomorphism of vector spaces with ordered bases.
Since such isomorphisms are unique,
the set-theoretic section $\tilde{s} : J \to W$ given by
$\tilde{s}(j) = w_j^{-1} \circ s(j)$ is a group homomorphism as desired.

Note that the section constructed only depends on the choice of
orderings for the bases within each $V_\lambda$.  These are permuted by
$W^\circ$, so the section is unique up to conjugacy.

We have established the theorem when $k$ is algebraically closed.
However, $U$, $\cox{T}$, $T$, $\GL_1(A)$, $W$, $W^\circ$, $J$
and the splittings can be all be defined over a general field $k$.
All the remaining statements follow since they are true over an
algebraic closure.
\end{proof}

\begin{remark}
As remarked above, it is unclear if $\Aut(X)$ is ever not a smooth group
scheme.
Regardless, there is a group scheme homomorphism $\tAut(X)/S \to \Aut(X)$
defined via the universal property of the categorical quotient
$\cox{X} \to X$,
which is an isomorphism on the level of $K$-points for any field $K$.
The theorem still holds without the smoothness assumption
if we replace $\Aut(X)$ with $\tAut(X)/S$.
\end{remark}

\begin{remark}
Note that in~\cite{Cox95The-homogeneous}, it is erroneously stated that
\emph{all} graded endomorphisms of $\Cox(X)$ form a (not necessarily
separable) algebra $B$.
This would lead to a description $\tAut(X) \simeq \GL_1(B) \rtimes J$ above.
This error was first addressed in~\cite{Cox14Erratum}.
In fact, the graded endomorphisms form an algebra if and only if
the unipotent radical $U$ is trivial.
Indeed, if $U$ is trivial then $B=A$.
However, if $U$ is non-trivial then there exists an element
$g \in B$ which acts as $g(x_i^*)=x_i^*$ for all $i \ne j$ and
$g(x_j^*) = x_j^* + M$ where $j$ is an integer in $1, \ldots, r$,
and $M$ is some monomial in $x_1^*, \ldots, x_r^*$ of degree $\ge 2$.
If $B$ were an algebra then $g$ would commute with scalar
multiplication, but this fails since $\deg(M) \ne 1$.
\end{remark}

\begin{remark}
Let $V$ be a vector space of dimension $n$.
If $X$ is $\bbP(V)$, the set of $1$-dimensional subspaces of $V$,
then the Cox endomorphism algebra $A$ is simply $\End(V)$.
As a special case of the Severi-Brauer construction,
we may recover $X$ as the variety of right ideals of $A$
(see \cite{Sal99Lectures} or \cite{KnuMerRos98The-book}).
\end{remark}

\section{Twists of the split toric variety}
\label{sec:twists}

Throughout this section, we assume that $X$ is a split projective toric
$T$-variety with a fixed torus embedding $T \hookrightarrow X$ as in the
previous section.
The projectivity assumption is to ensure that descent is
effective.  This assumption can be often be weakened
(see \cite{Hur11Toric} for a complete description of when descent is
effective for toric $T$-varieties.)
Moreover, we also assume Hypothesis~\ref{assume} holds.

Note that the Galois cohomology set $H^1(k,G)$ is functorial in the
group $G$.  Thus, we may apply the functor $H^1(k,-)$ to the commutative
diagram \eqref{eq:twoLines}.  The following theorem investigates the
result of this operation:

\begin{thm} \label{thm:mainSquare}
The commutative square
\begin{equation} \label{eq:rawSquare}
\xymatrix{
H^1(k, \cox{T} \rtimes W) \ar@{->}[r] \ar@{->}[d] &
H^1(k, T \rtimes W) \ar@{->}[d] \\
H^1(k, \tAut(X)) \ar@{->}[r] &
H^1(k, \Aut(X)) \\
} \ .
\end{equation}
obtained from the commutative diagram \eqref{eq:twoLines}
is canonically isomorphic to the square
\begin{equation} \label{eq:mainSquare}
\xymatrix{
H^1(k, W) \ar@{^{(}->}[r] \ar@{->>}[d] &
H^1(k, T \rtimes W) \ar@{->>}[d] \ar@{-->}@/_1pc/[l] \\
H^1(k, J) \ar@{^{(}->}[r] \ar@{-->}@/_1pc/[u] &
H^1(k, \Aut(X)) \ar@{-->}@/_1pc/[l] \\
}
\end{equation}
where the downward maps are surjective and the rightward maps are
injective.
Moreover, the rightward maps have canonical retracts and the
left downward map has a canonical section.
\end{thm}

Before proving the theorem, we make some remarks.
Recalling the categories from Section~\ref{sec:prelimToric} we find:
\[
\Aut_\calR(X) = \Aut(X) \ , \quad \quad
\Aut_\calN(X) = T \rtimes W \ , \quad \quad
\Aut_\calW(X) = W \ .
\]
Thus, three of the coefficient groups appearing in \eqref{eq:mainSquare}
are simply the automorphism groups of toric varieties within
these categories.  The Galois cohomology sets represent the forms of
$X$ within each category.  As $\calW$ contains only neutral toric
varieties, we recover the interpretation from
\eqref{eq:mainSquareInterp}.

From this diagram, we recover the well-known fact that there is a unique
isomorphism class of a split toric variety (resp. split toric $T$-variety)
among all the possible $k$-forms.
For a given toric variety $X$ we call the unique split variety
\emph{the associated split toric variety} and denote it by $X_\Split$.

From the canonical retracts, we see that every toric variety $X$
has an \emph{associated neutral toric variety}, or
\emph{neutralization}, which we denote by $X_\neut$.
We call the set of forms which have a common neutralization
a \emph{neutralization class} and remark that these partition
the isomorphism classes of forms of $X$.

The canonical section provides every neutral toric variety with a
canonical isomorphism class of torus among the many which may act on the
variety.
When the toric variety is split,
the canonical section recovers the split torus.

\begin{remark}
An investigation of the top row of Theorem~\ref{thm:twoLines} using
Galois cohomology was also carried out in
\cite{EliLimSot14Arithmetic}.
\end{remark}

\begin{remark}
In \cite{VosKly84Toric} and \cite{MerPan97K-theory},
the neutralization is defined when $T$ is fixed
and is called the ``associated toric $T$-model.''
The theorem above shows that
one can define the neutralization independently of the particular
torus action chosen.
\end{remark}

We now prove Theorem~\ref{thm:mainSquare}.
We begin with some technical lemmas:

\begin{lem} \label{lem:semidirectInject}
Let $A$ and $C$ be algebraic groups where $C$ acts on $A$.
The map $H^1(k,A \rtimes C) \to H^1(k,C)$ is surjective
and the map $H^1(k,C) \to H^1(k,A \rtimes C)$ is injective.
\end{lem}

\begin{proof}
Apply the functor $H^1(k,-)$ to the composition
$C \to A \rtimes C \to C$.
\end{proof}

\begin{lem} \label{lem:semidirectGC}
Let $A$ and $C$ be algebraic groups where $C$ acts on $A$.
Let $\xi$ be a set of cocycle representatives for the
set $H^1(k,C)$.
There is a canonical bijection
\[
H^1(k,A \rtimes C) \simeq
\coprod_{c \in \xi} H^1(k,{}_c A) / H^0(k,{}_c C) \ ,
\]
functorial in $A$,
where each component of the disjoint union is a fibre of the
canonical map to $H^1(k,C)$.
In particular, $H^1(k,{}_c A)$ is trivial for every $c \in \xi$
if and only if the map $H^1(k,A \rtimes C) \to H^1(k, C)$
is a bijection.
\end{lem}

\begin{proof}
The last statement follows from the previous statements since the
indexing set $\xi$ is bijective with $H^1(k,C)$. 

By Lemma~\ref{lem:semidirectInject}, there is a canonical choice of
preimage in $H^1(k,A \rtimes C)$ for elements in $H^1(k,C)$.
By Corollary 2 of \S 5.5 of \cite{Ser02Galois} 
we obtain the desired bijection.
More precisely, for $c \in \xi$, we obtain a map
$H^1(k,{}_cA) \to H^1(k,{}_c(A \rtimes C))$
whose image is $H^1(k,{}_cA)/H^0(k,{}_c C)$;
then we use the bijection $\tau_c : H^1(k,A \rtimes C) \to
H^1(k,{}_c(A \rtimes C))$ as in \S 5.4 of \cite{Ser02Galois}.
(Note that this bijection is canonical \emph{up to the choice of cocycle
representatives} $\xi$.)

To prove functoriality, consider another algebraic group $B$ with an
action of $C$ and a $C$-equivariant homomorphism $A \to B$.
The morphism $f \rtimes C : A \rtimes C \to B \rtimes C$ gives rise to
a commutative diagram
\[
\xymatrix{
H^1(k,A \rtimes C) \ar@{->}[r] \ar@{->}[d] &
H^1(k,B \rtimes C) \ar@{->}[d] \\
H^1(k, C) \ar@{=}[r] &
H^1(k, C) \\
} \ .
\]
Thus, the map takes fibres to fibres.
Note that twisting by $c \in \xi$ is a functorial operation,
so we can twist $f$ by $c$ to obtain ${}_c f : {}_c A \to {}_c B$.
The map
\[
\coprod_{c \in \xi} H^1(k,{}_c A) / H^0(k,{}_c C) \to
\coprod_{c \in \xi} H^1(k,{}_c B) / H^0(k,{}_c C)
\]
is obtained by taking a disjoint union of quotients by $H^0(k,{}_c C)$.
This process preserves compositions and the identity.
\end{proof}

\begin{lem} \label{lem:AoverSsimplify}
Let $A$ be a separable $k$-algebra with center $Z(A)$.
Let $S$ be a closed subgroup of $\GL_1(Z(A))$ and suppose $J$ is an
algebraic subgroup of $\Aut(A)$ which stabilizes $S$.
Then the induced morphism of Galois cohomology sets
\[ H^1(k, \GL_1(A) \rtimes J) \to H^1(k, \GL_1(A)/S \rtimes J) \]
is canonically isomorphic to the injection
\[ H^1(k, J) \hookrightarrow H^1(k, \GL_1(A)/S \rtimes J) \]
and has a canonical retract.
\end{lem}

\begin{proof}
Since $J$ acts on the algebra $A$ by automorphisms,
for any cocycle $c \in Z^1(k,J)$ we have
${}_c\GL_1(A) = \GL_1({}_cA)$ by Proposition~\ref{prop:cGLisGLc}.
Thus, by Hilbert 90, we see that $H^1(k,{}_c\GL(A))$ is trivial for any
cocycle $c \in Z^1(k,J)$.

By Lemma~\ref{lem:semidirectGC}, we conclude that
\[
H^1(k, \GL_1(A) \rtimes J) \simeq H^1(k, J) \ .
\]
We obtain injectivity and the canonical retract from
Lemma~\ref{lem:semidirectInject}.
\end{proof}

We now observe that, from the perspective of Galois cohomology,
the unipotent radical is irrelevant:

\begin{prop} \label{prop:ignoreU}
There are canonical bijections
\[ H^1(k,\tAut(X)) \simeq H^1(k, \GL_1(A) \rtimes J ) \]
and 
\[ H^1(k,\Aut(X)) \simeq H^1(k,(\GL_1(A)/S) \rtimes J) \ . \]
\end{prop}

\begin{proof}
Since $X$ is split, $X$ can be defined over the prime field $k_0$ of
$k$; in other words, $X = Y_k$ for a split toric variety $Y$ over $k_0$.
All prime fields are perfect, thus the unipotent radicals of
$\tAut(Y)$ and $\Aut(Y)$ are $k_0$-split.
As a consequence, the unipotent radicals of $\tAut(X)$ and $\Aut(X)$
are $k$-split and, thus, we may apply Lemma 7.20 of~\cite{GilMor13Actions}.
\end{proof}

\begin{proof}[Proof of Theorem~\ref{thm:mainSquare}]
First we consider the top row.
Recall that $\cox{T}$ is a split torus and thus is $\GL_1(E)$ for a
split \'etale $k$-algebra $E$.  The group $W$ permutes a basis for $V$
and thus acts on $E$ by algebra automorphisms.
From Lemma~\ref{lem:AoverSsimplify},
we obtain injectivity with a canonical retract.

By Proposition~\ref{prop:ignoreU},
we may assume the bottom row in \eqref{eq:rawSquare} is
\[
H^1(k,\GL_1(A) \rtimes J) \to H^1(k,(\GL_1(A)/S) \rtimes J) \ .
\]
Again, Lemma~\ref{lem:AoverSsimplify}
provides injectivity with a canonical retract.

Now, we show that the vertical maps are surjections.
For any cocycle $c \in Z^1(k,J)$ the map
\[ {}_c(\cox{T} \rtimes W^\circ) \to {}_c\GL_1(A) \]
is the inclusion of the normalizer of a maximal torus into a connected
algebraic group and thus the induced map
\[ H^1(k,{}_c(\cox{T} \rtimes W^\circ)) \to H^1(k,{}_c\GL_1(A)) \]
is surjective by Corollary 5.3 of~\cite{CheGilRei08Reduction}
(Lemma III.4.3.6 of~\cite{Ser02Galois} when $k$ is perfect).
Via Lemma~\ref{lem:semidirectGC}, we conclude that the left vertical map
in \eqref{eq:rawSquare} is surjective.  Similarly, we conclude the right
vertical map is surjective.

Since $J$ and $W$ are finite constant groups, the elements of
$H^1(k,J)$ and $H^1(k,W)$ are simply homomorphisms
from $\Gamma_k$ to $J$ or $W$ up to conjugacy.
By Theorem~\ref{thm:twoLines}, there exists a section of the map
$W \to J$ which is unique up to conjugacy.
Thus, the induced map $H^1(k,J) \to H^1(k,W)$ is independent of the
choice of section $J \to W$ and, thus, is canonical. 
\end{proof}

\begin{example}[ $\bbP^1_\bbR \times \bbP^1_\bbR$ ] \label{ex:P1P1}
Consider $X = \bbP^1 \times \bbP^1$ over the real numbers $\bbR$.
There is precisely one other $\bbR$-form $C$ of $\bbP^1$ over $\bbR$
corresponding to the subvariety
\[ x^2 + y^2 + z^2 = 0 \]
cut out of $\bbP^2$.  There are only two isomorphism classes of
$1$-dimensional tori over $\bbR$ which we denote by $\bbR^\times$ and $S^1$.

Here $W \simeq D_8$, and $T \simeq \bbG_m^2$.  The group $J \simeq C_2$
can be thought of as the group which interchanges the two fibrations
$X \to \bbP^1$.
In this case, one can compute all the cohomology groups in
Theorem~\ref{thm:mainSquare} explicitly.  We summarize the conclusions
of this computation.

There are two neutralization classes.  The first contains:
\begin{enumerate}
\item[(a)] $\bbP^1 \times \bbP^1$ with possible tori
$\bbR^\times \times \bbR^\times$, $\bbR^\times \times S^1$
and $S^1 \times S^1$,
\item[(b)] $\bbP^1 \times C$ with possible tori
$\bbR^\times \times S^1$ and $S^1 \times S^1$,
\item[(c)] $C \times C$ with torus $S^1 \times S^1$.
\end{enumerate}
And the second class contains:
\begin{enumerate}
\item[(d)] $\Weil{\bbC/\bbR}\bbP^1_\bbC$ with torus
$\Weil{\bbC/\bbR}\bbC^\times$.
\end{enumerate}
We point out the role of $J$ in these computations using
Lemma~\ref{lem:semidirectInject}.
First, the action of $H^0(k,{}_cJ)$ is necessary to
establish that $\bbP^1 \times C$ and $C \times \bbP^1$ are isomorphic.
Second, the necessity of non-trivial cocycles $c \in Z^1(k,J)$
is witnessed by the existence of $\Weil{\bbC/\bbR}\bbP^1_\bbC$.
\end{example}

\section{Properties of Twisted Forms}
\label{sec:twistedProps}

Throughout this section, $X$ is a projective toric variety
which is \emph{not necessarily} split.
However, we assume the split form $X_\Split$ of $X$
satisfies Hypothesis~\ref{assume} (which then also holds for $X$).

Clearly, one may define the automorphism group $\Aut(X)$ and its unipotent
radical $U$.
As before, we define the algebraic group $S$ as the dual of the $\Gamma_k$-module
$\Cl(\overline{X})$.
We define the \emph{group of class group automorphisms}
$J$ as the image of the morphism
$\Aut(X) \to \Aut(S)$ as before.
The group $J$ is a finite algebraic group, but may not be constant.
If we do specify a torus $T$, then we define $W$ as the Weyl group of
$T$ in $\Aut(X)$; in other words $W := N_{\Aut(X)}(T)/T$.
All of these objects agree with the corresponding objects of
the split form over a separable closure.

The theory of Cox rings is more complex for non-split toric varieties.
Depending on one's definitions, they may not exist or they may not be
polynomials rings (see \cite{DerPie14Cox-rings}).
Despite this, one may still define a Cox endomorphism algebra for a general
toric variety.

\begin{prop} \label{prop:assocA}
There exists an algebra $A$ and a morphism $\GL_1(A) \to \Aut(X)$
which coincides with the split case over the separable closure.
The isomorphism class of the center $Z(A)$ of $A$ is determined only by
the neutralization class of $X$.
\end{prop}

We call the algebra $A$ in the proposition above the
\emph{Cox endomorphism algebra of $X$}, generalizing
Definition~\ref{def:CoxAlg}.

\begin{proof}
We use the fact that $X$ is a $k$-form of $X_\Split$.

Recall that the algebra $A_\Split$ associated to $X_\Split$
comes with a map $\GL_1(A_\Split) \to \Aut(X_\Split)$.
From Theorem~\ref{thm:twoLines}, we have a morphism
\[ \Aut(X_\Split) \to \Aut(A_\Split) \]
since $S_\Split$ is a subgroup of
the group $\GL_1(Z(A_\Split))$.
Thus the morphism $\GL_1(A_\Split) \to \Aut(X_\Split)$ is
$\Aut(X_\Split)$-equivariant, and we obtain the
desired algebra $A$ and map by twisting.

By composition, there is also a map
$\Aut(X_\Split) \to \Aut(Z(A_\Split))$
which factors through $J_\Split$ since 
$\GL_1(A_\Split)$ centralizes $Z(A_\Split)$.
Thus, the induced morphism
\[ H^1(k, \Aut(X_\Split)) \to H^1(k,\Aut(Z(A_\Split))) \]
then factors through $H^1(k,J_\Split)$.
Thus, the isomorphism class of $Z(A)$ depends only on the neutralization
class of $X$.
\end{proof}

\begin{example}[Severi-Brauer varieties] \label{ex:SBalgebra}
Given a central simple algebra $A$, one can construct the Severi-Brauer
variety $X$ associated to $A$.
Proposition~\ref{prop:assocA} simply reverses this process.
In the case of Severi-Brauer varieties,
the isomorphism class of the toric variety $X$ is
completely detemined by the isomorphism class of its Cox
algebra $A$.
\end{example}

\begin{example}[del Pezzo of degree $6$] \label{ex:DP6}
Let $X$ be the split del Pezzo surface of degree $6$.  As a toric
variety, its effective $T$-invariant divisors are spanned by
$6$ elements which correspond to the blow ups of three
points on $\bbP^2$ and the strict transforms of the lines between them.

There is a common convention for the coordinates of the rays of this surface.
Using this convention, the map $\cox{N} \to N$ is given by
\[
\begin{pmatrix}
1 & 0 & -1 & -1 & 0 & 1 \\
0 & 1 & -1 & 0 & -1 & 1 \\
\end{pmatrix} \ .
\]
By duality we obtain a map
$\cox{M} \to \Pic(X)$ given by
\[
\begin{pmatrix}
0 & 0 & 0 & 1 & 1 & 1\\
1 & 0 & 0 & 0 & -1 & -1\\
0 & 1 & 0 & -1 & 0 & -1\\
0 & 0 & 1 & -1 & -1 & 0\\
\end{pmatrix}
\]
where we denote the ordered basis of $\Cl(\overline{X}) \simeq \Pic(X)$
by $H$, $E_1$, $E_2$, $E_3$.

Note that as all $6$ rays have different images in $\Cl(\overline{X})$,
all the spaces $V_\lambda$ are $1$-dimensional.
Thus, the Cox endomorphism algebra $A$ is an \'etale algebra and thus equal to
its center.
By Proposition~\ref{prop:assocA},
the algebras associated to the $k$-forms of $X$ are determined only by
the neutralization class.

However, not all forms of $X$ are neutral.
Thus, the Cox endomorphism algebras $A$ do not suffice to distinguish $k$-forms of
$X$ in this case.
However, in~\cite{Blu10Del-Pezzo}, a different pair of separable algebras
is associated to each $k$-form of $X$ which do distinguish all
isomorphism classes of $X$.
We investigate this pheonomenon in Section~\ref{sec:approximation} below;
in particular, see Example~\ref{ex:blunk}.
\end{example}

\begin{remark}[Maximal \'etale algebras]
For Severi-Brauer varieties, there is a bijective correspondence between
maximal \'etale subalgebras of $A$ and maximal tori of $\Aut(X)$.
This holds for toric varieties when the unipotent radical $U$ is
trivial.
Conditions for when the hypothesis $U=0$ holds are investigated
in~\cite{Nil06Complete}.
\end{remark}

In Cox's original paper, the Cox ring was only defined for split toric
varieties (over $\bbC$).
Cox rings can be extended to more general varieties as a
ring structure on the direct sum of the vector spaces $H^0(X,\calO_X(D))$
for all $D \in \Cl(X)$.
This definition makes sense when the base field is not algebraically closed,
but we lose the connection with universal torsors.
An alternate definition was proposed in \cite{DerPie14Cox-rings}, which
is better suited to applications over non-closed fields.
However, like universal torsors,
these Cox rings do not always exist, nor are they necessarily unique.
The following is in the spirit of the latter definition.

Recall that in the split case, the \emph{characteristic space} of a
toric variety $X$ is the $S$-invariant open subset $\cox{X}$ of $V$
along with the categorical quotient map $\psi : \cox{X} \to X$.
We define a \emph{characteristic space} of a (not necessarily split)
toric variety $X$ as a subvariety $\cox{X}$ of an affine space $V$ along
with a map $\psi : \cox{X} \to X$ which coincides with the
characteristic space of the split toric variety over the separable
closure.
When $\cox{X}$ does exist, we define $\tAut(X)$ to be normalizer of $S$ in
$\Aut(\cox{X})$ as before.

\begin{prop}
There exists a characteristic space $\psi : \cox{X} \to X$ if
and only if $X$ is neutral.
When $\cox{X}$ exists, we have a decomposition
\[ \tAut(X) \simeq U \rtimes \GL_1(A) \rtimes J \]
as in the split case.
\end{prop}

\begin{proof}
Assume a characteristic space exists.  Then we have a
dominant rational map $V \to X$ where $V$ is an affine space.
Thus $X$ has a Zariski-dense set of $k$-points.  We conclude that $X$ is
neutral since any open $T$-orbit for any $T$-action must contain a
$k$-point.

Now, assume that $X$ is neutral.
By Theorem~\ref{thm:mainSquare}, we may choose a cocycle $c$
in $Z^1(K,J_\Split)$ such that $X \simeq {}_c(X_\Split)$.
Since $J_\Split$ maps to $\tAut(X_\Split)$ and leaves $S_\Split$ stable,
we may define an affine space $V := {}_c(V_\Split)$
containing a characteristic space
$\cox{X} := {}_c(\cox{X}_{\Split})$
both with $S$-actions.
We define the map $\psi \colon \cox{X} \to X$ as the twist
${}_c(\psi_\Split)$.

The splitting $J_\Split \to \tAut(X_\Split)$ is $J_\Split$-equivariant
and thus, when we twist by $c$, the morphism $\tAut(X) \to J$ splits
and we have the desired decomposition.
\end{proof}

\begin{remark}
Note that the map $\psi \colon \cox{X} \to X$ is not unique as an
$S$-scheme over $X$ --- even up to isomorphism.
Indeed, when $X$ is smooth, $\psi$ is a
universal torsor of $X$; their isomorphism classes are in bijection
with $H^1(k,S)$ (see \S 2 of~\cite{ColSan87La-descente}).
\end{remark}

\begin{remark}[Canonical torus]
When $X$ is neutral, we can define $\cox{T}$ with an action on $\cox{X}$
as in the split case.
In this case, there is a canonical choice
for the isomorphism class of the torus $\cox{T}$ (and thus $T$).
This follows from the fact that the morphism
$H^1(k,W_\Split) \to H^1(k,J_\Split)$ has a canonical section.
Specifically, as $A$ is a neutral separable $k$-algebra,
it is a product of matrix algebras
\[ A = M_{n_1}(F_1) \times \cdots \times M_{n_r}(F_r) \]
where $F_1, \ldots, F_r$ are separable field extensions of $k$
and $n_1, \ldots, n_r$ are positive integers.
There is a maximal \'etale subalgebra of $A$ of the form
\[ E = (F_1)^{n_1} \times \cdots \times (F_r)^{n_r} \ . \]
The canonical torus $\cox{T}$ is then $\GL_1(E)$.
\end{remark}

\begin{remark}[Restricted automorphism groups]
\label{rem:restrictAuto}
Let $X$ be a neutral toric variety.
Given a finite algebraic subgroup $I$ of $J$, the \emph{$I$-restricted
automorphism group of $X$}, denoted $\Aut_I(X)$, is the preimage of $I$
in $\Aut(X)$.
Note that, in particular, $\Aut_1(X) = \Aut(X)^\circ$.  The Galois
cohomology set $H^1(k,\Aut_1(X))$ can be interpreted as the set of
isomorphism classes of toric varieties $Y$ isomorphic to $X$
along with a fixed isomorphism $\Pic(Y) \simeq \Pic(X)$.
Note that if we drop the explicit isomorphism, we obtain the set
\[ H^1(k,\Aut_1(X)) / H^1(k,J) \]
which is the set of isomorphism classes of toric varieties $Y$ with a
neutralization isomorphic to $X$.
\end{remark}

\begin{example}[Products of projective spaces]
\label{ex:prodProj}
Let $X_\Split = (\bbP^{q-1})^n$ be a product of projective spaces
given two integers $q,n > 1$.
Here $\Pic(X_\Split) \simeq \bbZ^n$, $S \simeq \bbG_m^n$ and $J \simeq S_n$.
In this case, \eqref{eq:twoLines} is as follows:
\begin{equation*}
\xymatrix{
1 \ar@{->}[r] &
\bbG_m^n \ar@{->}[r] \ar@{=}[d] &
(\bbG_m^q)^n \rtimes (S_q \wr S_n) \ar@{->}[r] \ar@{->}[d] &
(\bbG_m^q/\bbG_m)^n \rtimes (S_q \wr S_n) \ar@{->}[r] \ar@{->}[d] &
1 \\
1 \ar@{->}[r] &
\bbG_m^n \ar@{->}[r] &
\GL_m \wr S_n \ar@{->}[r] &
\PGL_m \wr S_n \ar@{->}[r] &
1 \\
}
\end{equation*}
where $G \wr S_n$ is the wreath product $G^n \rtimes S_n$
for a group $G$.

Here the algebra $A_\Split$ is simply $M_q(k)^n$ and thus
the automorphism group $\Aut(X_\Split) \simeq \PGL_q \wr S_n$ of the
variety is isomorphic to $\Aut(A_\Split)$.
Thus, forms $X$ of $X_\Split$ correspond to forms $A$ of $A_\Split$.
The neutral forms $X$ correspond to neutral forms $A$, which,
in turn, correspond to \'etale $k$-algebras of degree $n$.

More geometrically, $k$-forms $X$ are products of Weil
restrictions
\[
\Weil{F_1/k}X_1 \times \cdots \times \Weil{F_r/k}X_r
\]
where each $X_i$ is a Severi-Brauer variety of dimension $q-1$ over
the field $F_i$.  When $X$ is neutral, each $X_i$ is a projective space.
\end{example}

\section{Non-abelian $H^2$}
\label{sec:non-abelianH2}

In this section, we extend the long exact sequences in Galois cohomology 
from Theorem~\ref{thm:twoLines} to $H^2$.  Unlike the situation for
Severi-Brauer varieties, the original sequences do not correspond to
central extensions, so we cannot use the ordinary abelian cohomology group
$H^2(k,S)$.

In \cite{Spr66Nonabelian} and
\S IV.4.2 of \cite{Gir71Cohomologie}, the ordinary sequence from Galois
cohomology is extended to a non-abelian version of $H^2$
(see also more recent work in \cite{Bor93Abelianization},
\cite{FliSchSuj98Grothendiecks} and \cite{Flo04Zero-cycles}).
Our applications are less ambitious, so we have the luxury
of a simpler exposition.
Rather than consider the abelian group $H^2(k,S)$,
we consider $H^2(k,S \to J)$ which has the structure of
a set with a distinguished subset.
Here the notation for the ``coefficients'' is meant to suggest a
\emph{crossed module} as in, for example, \cite{Bre10Notes}.

Let $S$ be a smooth split $k$-group of multiplicative type and let $J$ be a
finite subgroup of $\Aut(S)$.  Note that $J$ is a constant
group.

We have a natural action of $J$ on each cocycle
$c \in Z^1(k,J)$ via $j(c)_\sigma := j c_\sigma j^{-1}$
for all $\sigma \in \Gamma_{k}$.
Two cycles are cohomologous if and only if they are in the same orbit
under this action (essentially by definition).

Given a cocycle $c \in Z^1(k,J)$ and an element $j \in J$ we have
an isomorphism
\[ j_* \colon H^2(k,{}_cS) \to H^2(k,{}_{j(c)}S) \]
which is defined on cocycles $s \in Z^2(k,{}_cS)$ via
\[ j_*( s)_{\sigma,\tau} := j(s_{\sigma,\tau})
\textrm{ for all } \sigma,\tau \in \Gamma_{k} . \]
One checks that the image cocycle sits in $Z^2(k,{}_{j(c)}S)$
as expected and that cohomology classes are preserved.

We now define the main object of study in this section:

\begin{defn} \label{def:H2}
We define the following set:
\[
H^2(k,S \to J) := \left.
\left(\coprod_{c \in Z^1(k,J)} H^2(k,{}_cS) \right) \right/ J \ .
\]
We define the set of \emph{neutral elements}
as the subset of $H^2(k,S \to J)$ containing
the trivial elements in each component $H^2(k,{}_cS)$.
This endows the set $H^2(k,S \to J)$ with the structure of a set
with a distinguished subset.
The image of the trivial element from $H^2(k,S)$ will be referred to as
the \emph{trivial element} of $H^2(k,S \to J)$.
\end{defn}

Now we define the connecting homomorphism.
Suppose $\cox{G}$ and $G$ are algebraic groups sitting in an exact
sequence
\begin{equation} \label{eq:exampleSeq}
1 \to S \to \cox{G} \to G \to 1
\end{equation}
such that the conjugation action of $G$ on $S$ induces a surjection
$\pi : G \to J$.

Let $a$ be a cocycle in $Z^1(k,G)$.
Let $\cox{a} : \Gamma_k \to \cox{G}(k_s)$ be a continuous function
lifting $a$ (which always exists since $\Gamma_k$ is profinite
and $G(k_s), \cox{G}(k_s)$ have the discrete topology).
Define a function
$\Delta \cox{a} \colon \Gamma_k \times \Gamma_k \to \cox{G}(k_s)$
via
\[ (\Delta \cox{a})_{\sigma,\tau} :=
\cox{a}_{\sigma} ({}^{\sigma} \cox{a}_{\tau}) (\cox{a}_{\sigma \tau})^{-1} \]
for all $\sigma, \tau \in \Gamma_k$.
We will see that $\Delta \cox{a}$ is a cocycle in $Z^2(k,{}_{\pi(a)}S)$,
and that this gives rise to a well-defined map
\[ \delta \colon H^1(k,G) \to H^2(k, S \to J) \]
which we call the \emph{connecting map}.

\begin{lem} \label{lem:agreesWithSpringer}
The connecting map
\[ \delta \colon H^1(k,G) \to H^2(k, S \to J) \]
is well-defined and canonically isomorphic as sets with distinguised
subsets to the map
\[ \delta^1 \colon H^1(k, \cox{G},S)
\to H^2(k, S \ \mathrm{rel} \ \cox{G}) \]
from 1.20 of~\cite{Spr66Nonabelian}.
\end{lem}

\begin{proof}
We will prove the isomorphism with Springer's map;
the fact that our construction is well-defined then follows for free.

Let $a$ be a cocycle in $Z^1(k,G)$ and let $\cox{a} : \Gamma_k \to
\cox{G}(k_s)$
be a continuous function lifting $a$.
In Springer's notation, the set $Z^1(k, \cox{G}, S)$ is simply the set
of continuous functions $b : \Gamma_k \to \cox{G}(k_s)$ lifting cocycles
in $Z^1(k,G)$.  We may simply assume $\cox{a} = b$.
To each cocycle $b$ there is an associated
$2$-cocycle $(f,g)$ for $H^2(k, S \ \mathrm{rel} \ \cox{G})$
where
\begin{align*}
f_{\sigma}(s) &= b_{\sigma} ({}^{\sigma}\! s) (b_{\sigma})^{-1}\\
g_{\sigma,\tau} &=
b_{\sigma} ({}^{\sigma} b_{\tau}) (b_{\sigma \tau})^{-1}
\end{align*}
for $s \in S(k_s)$ and $\sigma, \tau \in \Gamma_k$.
Comparing this to Definition~\ref{def:H2},
we will see that $f$ corresponds to the choice of cocycle $c$
and $g$ corresponds to a cocycle in $Z^2(k,{}_cS)$.

This cocycle $(f,g)$ sits
naturally inside $Z^2(k,S,\lambda_a)$ where $\lambda_a$ is the
$\Gamma_k$-kernel induced by the automorphisms $f_{\sigma}$ of $S$.
Since $S$ is abelian, the kernels $\lambda_a$ are honest continuous maps
$\Gamma_k \to \Aut(S(k_s))$ (where $\Aut(S(k_s))$ has the discrete topology).
Thus, the kernel $\lambda_a$ coincides with the cocycle
$c=\pi_*(a)$ in $Z^1(k,J)$.
The cocycle $g$ is equal to $\Delta b$ as functions
and they are then $2$-cocycles in $Z^2(k,{}_cS)$.

The normalizer $N$ of $S$ in $\cox{G}$ is $\cox{G}$ itself.
Given an element $n \in N$ we have the image $j=\pi(n) \in J$.
The induced action of $n$ on the set of $\Gamma_k$-kernels of $S$
coincides with the induced action of $j$ on $Z^1(k,J)$ defined above,
and the induced action of $n$ on the pairs $(f,g)$ coincides with the
action of $j$ on $2$-cocycles in the sets $Z^2(k,{}_cS)$.
Thus the constructions agree.
\end{proof}

Now, suppose that $\cox{G} = R \rtimes J$ where $R$
is an algebraic group containing $S$ as a $J$-stable central closed
subgroup.
We have a $J$-equivariant exact sequence
\[ 1 \to S \to R \to R/S \to 1 \]
corresponding to a central extension.

\begin{lem} \label{lem:altH2desc}
Let $\xi$ be a set of cocycle representatives for $H^1(k,J)$.
The connecting map
\[ \delta \colon H^1(k,G) \to H^2(k, S \to J) \]
is canonically isomorphic to the map
\[
\coprod_{c \in \xi} \left(
\frac{H^1(k,{}_c(R/S))}{H^0(k,{}_cJ)}
\right)
\to
\coprod_{c \in \xi} \left(
 \frac{H^2(k,{}_cS)}{H^0(k,{}_cJ)}
\right)
\]
induced from the usual connecting maps
$H^1(k,{}_c(R/S)) \to H^2(k,{}_cS)$.
\end{lem}

\begin{proof}
We have the decomposition of $H^1(k,G)$ from 
Lemma~\ref{lem:semidirectGC}.

The orbits of $J$ in $Z^1(k,J)$ are precisely the
cohomology classes.
Thus, the images of the maps $H^2(k,{}_cS) \to H^2(k,S \to J)$
over all cocycles $c \in \xi$ are mutually disjoint and jointly
surjective.
For a particular cocycle $c \in \xi$, an element $j \in J$
fixes $c$ if and only if $j$ is in $H^0(k,{}_cJ)$.
Thus, the image of the map $H^2(k,{}_cS) \to H^2(k,S \to J)$ is simply
the quotient by $H^0(k,{}_cJ)$.
This establishes the decomposition of $H^2(k,S\to J)$.

It remains to prove that the connecting maps agree.
Fix a cocycle $c \in \xi$.
Let $u$ be a cocycle in $Z^1(k,{}_c(R/S))$ and
let $a_\sigma = u_\sigma c_\sigma$ be a corresponding cocycle in
$Z^1(k,G)$.
The cocycles $u$ and $a$ map to the same element of $H^1(k,G)$
in the decomposition from Lemma~\ref{lem:semidirectGC}.
Select a lift $v: \Gamma_k \to {}_cR(k_s)$ of $u$ and note that
$b_\sigma = v_\sigma c_\sigma: \Gamma_k \to \cox{G}(k_s)$
is a lift of $a_\sigma$.

The cocycle $\Delta v$ in $Z^2(k,{}_cS)$ is constructed via
\[ (\Delta v)_{\sigma,\tau}
:= (v_{\sigma}) ({}^{\sigma'}\! v_{\tau})
(v_{\sigma \tau})^{-1} \]
using the twisted action
${}^{\sigma'}\!s = (c_{\sigma}) ({}^{\sigma}\! s) (c_{\sigma})^{-1}$
for $s \in S(k_s)$ and $\sigma \in \Gamma_k$.
The computation
\begin{align*}
(\Delta v)_{\sigma, \tau}
&=
(v_{\sigma})
[(c_{\sigma}) ({}^{\sigma} v_{\tau}) (c_{\sigma})^{-1}]
(v_{\sigma \tau})^{-1}\\
&=
(v_{\sigma}) (c_{\sigma})
({}^{\sigma} v_{\tau}) [({}^{\sigma}\! c_{\tau})
(c_{\sigma \tau})^{-1}]
(v_{\sigma \tau})^{-1}\\
&=
(b_{\sigma})
({}^{\sigma} b_{\tau})
(b_{\sigma \tau})^{-1}
=
(\Delta b)_{\sigma, \tau} 
\end{align*}
shows that this has the desired image in $H^2(k,S \to J)$.
\end{proof}

The above shows that there is a surjective map
$H^2(k,S\to J) \to H^1(k,J)$ with a unique neutral element in each
fiber.
Given an element $\alpha$ in the set $H^2(k,S \to J)$,
we define its \emph{neutral class} as the fiber of this map and we
define the \emph{neutralization} of $\alpha$ as the unique neutral
element therein.

Let $A$, $B$, $C$ respectively be sets with distinguished subsets
$A'$, $B'$, $C'$ respectively.  We say that the composition $g
\circ f$ of functions $f : A \to B$, $g : B \to C$ is \emph{exact} if
$\im(f) = g^{-1}(C')$.  We can now state the main theorem of this
section.

\begin{thm} \label{thm:H2main}
Let $X$ be a split projective toric variety
satisfying Hypothesis~\ref{assume}.
Applying Galois cohomology to the sequence \eqref{eq:twoLines}, we
extend \eqref{eq:mainSquare} to the commutative diagram
\begin{equation} \label{eq:mainSquareWithH2}
\xymatrix{
H^1(k, W) \ar@{^{(}->}[r] \ar@{->>}[d] &
H^1(k, T \rtimes W) \ar@{->}[r] \ar@{->>}[d]&
H^2(k, S \to J) \ar@{=}[d] \\
H^1(k, J)\ar@{^{(}->}[r] &
H^1(k, \Aut(X)) \ar@{^{(}->}[r] &
H^2(k, S \to J)\\
}
\end{equation}
where the rows are exact sequences of sets with a distinguished subset
and the bottom right horizontal map is injective.
\end{thm}

Note that exact sequences for sets with distinguished subsets behave
differently from exact sequences of abelian groups.  Thus it is
perfectly reasonable to have an exact sequence consisting of two
injections.
In the bottom row of the above diagram, the exactness and injectivity
imply that 
the distinguished subset of $H^2(k,S\to J)$ is in bijection with $H^1(k,J)$.

\begin{proof}
The top row is the case where $R = \cox{T} \rtimes W^\circ$;
the bottom, where $R = \GL_1(A)$.
The exactness follows from Proposition~1.28 of \cite{Spr66Nonabelian}.
This can also be checked directly using
Lemmas~\ref{lem:semidirectGC}~and~\ref{lem:altH2desc}.
The only statement that remains to be proved is the injectivity of the
map $H^1(k, \Aut(X)) \to H^2(k, S \to J)$.

Let $c$ be a cocycle representative of an element in $H^1(k,J)$.
Consider the central extension
\[ 1 \to {}_cS \to {}_c\GL_1(A) \to {}_c(\GL_1(A)/S) \to 1 \ . \]
Since this is a central extension, we obtain an exact sequence
of pointed sets
\[
H^1(k,{}_c\GL_1(A)) \to H^1(k,{}_c(\GL_1(A)/S)) \to H^2(k,{}_cS) \ .
\]
Note that ${}_c(\GL_1(A)/S)$ acts on ${}_c\GL_1(A)\simeq \GL_1({}_cA)$
by algebra automorphisms of ${}_cA$.
Thus, for any cocycle $d$ in
$Z^1(k,{}_c(\GL_1(A)/S))$, we see that
$H^1(k,{}_d({}_c\!\GL_1(A))) \simeq H^1(k,\GL_1({}_d({}_cA)))$
is trivial by Hilbert 90.
By the Corollary in \S 5.7 of \cite{Ser02Galois},
we conclude the map $H^1(k,{}_c(\GL_1(A)/S)) \to H^2(k,{}_cS)$ is injective.
Injectivity is preserved after taking quotients by $H^0(k,{}_cJ)$,
and the result follows from Lemma~\ref{lem:altH2desc}.
\end{proof}

\begin{remark}
Suppose $X$ and $X'$ are toric varieties with fixed torus actions.
Theorem~\ref{thm:H2main} allows one to determine whether the underlying
varieties are isomorphic by considering the map
$H^1(k, T \rtimes W) \to H^2(k, S \to J)$
instead of $H^1(k, T \rtimes W) \to H^1(k,\Aut(X))$.
The former involves only the cohomology of abelian and finite
constant groups, which may be easier to use in some applications.
This generalizes the common trick where one distinguishes Severi-Brauer
varieties (or central simple algebras) by their Brauer classes rather
than by their classes in $H^1(k,PGL_n)$.
\end{remark}

\begin{remark}(Torsors)
\label{rem:torsors}
Note that Lemma~\ref{lem:altH2desc} allows us to easily describe the set
$H^2(k,S \to J)$ using torsors instead of $1$-cocycles.
Indeed, if $\xi$ is a collection of torsors representing elements of
$H^1(k,J)$ then we have
\[
H^2(k,S \to J) \simeq
\coprod_{T \in \xi} \left(
 \frac{H^2(k,{}^TS)}{H^0(k,{}^TJ)}
\right) \ .
\]
Note that this agrees with
Proposition~4.2.5(ii)~of~\cite{Gir71Cohomologie}.
\end{remark}

\begin{remark}(Restriction)
One can define a restriction map
\[ \Res_k^K \colon H^2(k,S\to J) \to H^2(K,S_K \to J_K) \]
for any field extension $K/k$.
For each $J$-torsor $T$, we have a restriction map
\[ H^2(k,{}^TS) \to H^2(K,{}^{T_K}S_K) \]
which is equivariant with respect to the homomorphism
\[ H^0(k,{}^TJ) \to H^0(K,{}^{T_K}J_K) \ . \]
Noting that $T \mapsto T_K$ is the restriction morphism
\[ H^1(k,J) \to H^1(K,J) \]
of the indexing sets, we obtain a restriction map via
Remark~\ref{rem:torsors}.
One checks that the restriction maps are compatible and we may
think of $H^2(-,S \to J)$ as being a functor from the category of field
extensions $K/k$ to the category of sets with a distinguished subset.
\end{remark}

\begin{remark}[Period] \label{rem:period}
The set $H^2(k,S \to J)$ is not a group in general.
However, we may still consider powers of elements.  Indeed, any element
$\alpha \in H^2(k,S \to J)$ is represented by an element $\beta$ in
some group $H^2(k,{}_cS)$ for some cocycle $c \in Z^1(k,J)$.
For any integer $n$, we define $\alpha^n$ as the image of $\beta^n$.
Since $J$ acts on each component of the disjoint union by
group automorphisms, this is well-defined.
In particular, the notion of ``period'' still makes sense:
the \emph{period}
of $\alpha \in H^2(k,S \to J)$ is the minimal positive integer $n$ such
that $\alpha^n$ is neutral.
\end{remark}

\begin{remark}[Index] \label{rem:index}
There are at least two natural notions of the ``index'' of an element
$\alpha \in H^2(k,S \to J)$.  We define the
\emph{neutralizing index} (resp. \emph{splitting index}) of
$\alpha$ is the greatest common divisor of the degrees of all finite
extensions $L/k$ such that the restriction $\Res^L_k(\alpha)$ is
neutral (resp. trivial).
\end{remark}

\begin{example} \label{ex:P1P3}
Let $k$ be the real numbers $\bbR$.
Consider $X=\bbP^1 \times \bbP^1$ where $S = \bbG_m^2$ and $J=C_2$.
By Lemma~\ref{lem:altH2desc} we find that
\[ H^2(\bbR, S \to J) \simeq
\left(\Br(\bbR)^2/C_2 \right) \sqcup \Br(\bbC) \]
and see that every Brauer class corresponds to a form of $X$
from Example~\ref{ex:P1P1}.

In contrast, consider $X=\bbP^1 \times \bbP^3$ where
$S = \bbG_m^2$ and $J$ is trivial.
Here
\[ H^2(\bbR, S \to J) \simeq \Br(\bbR)^2 \]
and the $k$-forms are
\[ \bbP^1 \times \bbP^3, C \times \bbP^3,
\bbP^1 \times C', C \times C'\]
where $C$ and $C'$, respectively, are the non-split forms of
$\bbP^1$ and $\bbP^3$, respectively.

Note that in both these examples $S$ is the same.
However, the finite group $J$ is different in each case.
\end{example}

\begin{example}[Permutation Lattices]
\label{ex:permH2}
Let $\widehat{P}$ be an $S_n$-lattice with a basis $\{ p_1, \ldots, p_n \}$
permuted by $S_n$.  Let $P$ be the dual torus $D(\widehat{P})$.
We would like to understand the set $H^2(k,P \to S_n)$.

Recall that elements of $H^1(k,S_n)$ are in bijection with
$\Gamma$-actions on $\{ p_i \}$ and with \'etale $k$-algebras
of degree $n$.
Thus, for any cocycle $c \in Z^1(k,S_n)$, the twisted torus ${}_cP$ is simply a
Weil restriction $\Weil{E/k} \bbG_{m,E}$ where $E$ is the corresponding \'etale
$k$-algebra.

Writing $E = F_1 \times \cdots \times F_r$ as a decomposition of
field extensions of $k$ we have
\begin{align*}
H^2(k,{}_cP)
& \simeq H^2(k,\Weil{E/k} \bbG_m)
\simeq \prod_{i=1}^r H^2(k,\Weil{F_i/k} \bbG_m) \\
& \simeq \prod_{i=1}^r H^2(F_i,\bbG_m)
\simeq \prod_{i=1}^r \Br(F_i)
\end{align*}
for each cocycle $c \in Z^1(k,S_n)$.  Note that the action of
$H^0(k,{}_cS_n)$ permutes isomorphic field extensions in the
decomposition of $E$.

By the decomposition in
Lemma~\ref{lem:altH2desc}, we conclude that $H^2(k,P \to S_n)$ consists
of all sets
\[\{ (\alpha_1, F_1), \ldots, (\alpha_r, F_r) \}\]
of pairs, where each $F_i$ is a separable field extension of $k$
and each $\alpha_i$ is an element of $\Br(F_i)$,
satisfying the condition $\sum_{i=1}^r [F_i:k] = n$.
\end{example}

\begin{remark}(Elementary Obstruction)
The elements in $H^2(k,S \to J)$ are closely related to the ``elementary
obstructions'' from~\cite{ColSan87La-descente}.  Indeed, given a
$k$-form of split smooth projective toric variety,
its class in $H^2(k,S \to J)$ is neutral if and only if the elementary
obstruction is trivial.
In Proposition~1.3~of~\cite{Sko09Automorphisms},
A.~Skorobogatov shows that the images of
the elementary obstruction and the classes defined above
actually coincide in the set $H^2(k,S \to \Aut(S))$
where one must extend the definition above to consider the infinite
group $\Aut(S)$
(in this reference, the type of a universal torsor is only well-defined
up to an isomorphism of the Picard group).
\end{remark}

\section{Comparison to ordinary Brauer groups}
\label{sec:H2approx}

There is a natural bijection between the isomorphism classes of
Severi-Brauer varieties and their associated central simple algebras.
This is a stronger statement than merely identifying the classes of
these algebras in $H^2(k,\bbG_m)$.
In \cite{Blu10Del-Pezzo}, M.~Blunk constructs a pair of separable algebras
$B$ and $Q$ and exhibits a bijection with the $k$-forms of a del Pezzo
surface of degree $6$ and a subset of the $k$-forms of the algebra $B
\times Q$.

Both of these constructions are well-behaved under extension of the base
field.  In both cases we have a toric variety $X$ and a separable
$k$-algebra $B$ along with an injection
\begin{equation} \label{eq:isoInj}
\left\{
{\text{\parbox{3.75cm}{
\centering
\begin{tabular}{c}
isomorphism classes\\
of $k$-forms of $X$
\end{tabular}
}}}
\right\}
\longhookrightarrow
\left\{
{\text{\parbox{3.75cm}{
\centering
\begin{tabular}{c}
isomorphism classes\\
of $k$-forms of $B$
\end{tabular}
}}}
\right\}
\end{equation}
which is natural in the field $k$.
The remainder of the paper focuses on a partial answer to this question.

Recall that the automorphism group $\Aut(B)$ of a separable algebra $B$
is a linear algebraic group with connected component $\Aut(B)^\circ$.
Given a finite closed subgroup $J$ of the quotient $\Aut(B)/\Aut(B)^\circ$, the
\emph{$J$-restricted automorphism group of $B$}, denoted $\Aut_J(B)$, is
the preimage of $J$ in $\Aut(B)$.

For a finite group $J$, a finitely generated $J$-module is \emph{permutation}
if it is free as a $\bbZ$-module with basis permuted by $J$;
a $J$-module is \emph{invertible} if it is a
direct summand of a permutation $J$-module.
Note that, for an algebraic group $G$ over a field $k$,
one can view $H^1(-,G)$ as a functor from the category of
field extensions $K/k$ to the category of sets.

The following theorem is our main goal for the rest of the paper and will
be proved in Section~\ref{sec:approximation}. 

\begin{thm} \label{thm:algInj}
Let $X$ be a smooth projective split toric variety over a field $k$.
Let $J$ be the image of the homomorphism $\Aut(X) \to \Aut(\Pic(X))$.
Suppose that $\Pic(X)$ is an invertible $J$-lattice.
There exists a canonical separable algebra $B$ over $k$,
and a natural transformation
\begin{equation} \label{eq:introAlgInj}
H^1(-,\Aut(X) ) \to H^1(-,\Aut_J(B)) \ .
\end{equation}
which is injective for every field $K/k$.
\end{thm}

Note that $\Pic(X)$ is an invertible $\Gamma$-lattice
for all smooth projective toric surfaces by \cite{Vos67On-two-dimensional}.
Thus, the conclusion of Theorem~\ref{thm:algInj}
holds for all smooth projective split toric surfaces

When $\Aut_J(B) = \Aut(B)$, as is the case for the central simple
algebras associated to Severi-Brauer varieties and Blunk's algebra $B$,
the set $H^1(k,\Aut_J(B))$ can be
interpreted directly as isomorphism classes of $k$-forms of $B$.
When $J$ is trivial, the set $H^1(k,\Aut_1(B))$ is simply the set of
isomorphism classes of $k$-forms of $B$ which preserve a fixed labelling
of each central simple algebra in the product.
In the intermediate cases, $\Aut_J(B)$ preserves this labelling
up to certain symmetries.
Thus, Theorem~\ref{thm:algInj} can be interpreted as variant of our
desired injection \eqref{eq:isoInj} for a mildly weaker notion of
isomorphism.

In \cite{MerPan97K-theory}, A.~Merkurjev and I.~Panin construct a
``motivic category'' which contains both toric varieties and separable
algebras.  In that paper, which was a starting point for Blunk's work,
they also associate forms of separable algebras to forms of toric
varieties.
In Theorem~\ref{thm:MerkPanin}, we show that their method is,
up to Brauer equivalence, essentially the same our own.

\subsection{Morphisms of non-abelian $H^2$}

Throughout this section, $X$ will be a split smooth projective toric
variety.  Recall that this implies that
$\Cl(\overline{X}) \simeq \Pic(\overline{X})$ is a lattice
and thus that $S$ is a torus.
Note that $\Pic(\overline{X})$ has a trivial $\Gamma_k$-action,
but a non-trivial $J$-action; it will be convenient to view
$\widehat{S} = \Pic(\overline{X})$ as a $J$-lattice.

In this section we make precise to what degree the information contained in
$H^2(k,S \to J)$ can be captured by ``ordinary Brauer groups.''
To make this precise, we need to define a notion of morphisms between
non-abelian $H^2$ sets.
Unfortunately, defining ``functoriality'' for non-abelian $H^2$ is a
rather delicate problem (see, e.g., \cite{AldNoo09Butterflies.}).
However, in our restricted context there is a natural
induced morphism which is well-behaved enough for our applications.

\begin{defn} \label{def:H2functoriality}
Let $S$ and $P$ be split groups of multiplicative type with actions of finite
groups $J$ and $I$ respectively.
Suppose $(m,g)$ is a pair of group homomorphisms
$m : S \to P$ and $g : J \to I$ such that $m$ is
$J$-equivariant where the $J$-action on $P$ is given via $g$.
Define the \emph{induced map}
\[ m_* \colon H^2(k,S \to J) \to H^2(k,P \to I) \]
via the decomposition from Lemma~\ref{lem:altH2desc} and the induced
maps
\[ H^2(k,{}_cS) \to H^2(k,{}_{g_*(c)}P) \]
for every cocycle $c$ representing a cohomology class in $H^1(k,J)$.
One checks that $m_*$ is compatible with restriction, and thus may be
viewed as a natural transformation
\[ m_* \colon H^2(-,S \to J) \to H^2(-,P \to I) \]
of functors from field extensions $K/k$ to the category of sets with
distinguished subsets.
\end{defn}

A $G$-lattice $M$ is \emph{invertible} if $M$ is a direct
summand of a permutation $G$-module.
Recall that a $k$-variety $Y$ is \emph{retract rational} if there exists an
affine space $V$, a dominant rational map $\psi : V \dasharrow X$ and a
rational map $\eta : X \dasharrow V$ such that the composition
$\psi \circ \eta$ is defined and equivalent to the identity as a rational
map.

\begin{thm} \label{thm:inj}
Let $X$ be a smooth projective split toric variety.
The following are equivalent:
\begin{enumerate}
\item[(a)]
There exists a morphism $\widehat{m} : \widehat{P} \to \widehat{S}$
of $J$-lattices, where $\widehat{P}$ is permutation,
such that the morphism of functors
\[ m_* : H^2(-, S \to J) \to H^2(-, P \to J) \]
is injective,
\item[(b)] $\widehat{S}$ is an invertible $J$-lattice,
\item[(c)] for every field extension $K/k$, every neutral $K$-form of $X_K$ is
retract rational.
\end{enumerate}
\end{thm}

\begin{remark}
Note that retract rationality is a birational invariant; thus,
the retract rationality of a toric variety can be determined by
consideration of only the open orbit.
Retract rational tori have been completely classified in small
dimensions.  See \cite{Vos67On-two-dimensional},
\cite{Kun87Three-dimensional} and \cite{HosYam12Rationality0}.
In fact, all $2$-dimensional tori are rational; thus all toric surfaces
satisfy the equivalent conditions of Theorem~\ref{thm:inj}.
\end{remark}

We may interpret Theorem~\ref{thm:inj} as a statement about ordinary
Brauer groups in the following way.
Fix a choice of basis $\omega$ for $\widehat{P}$.
Then twisting by a cocycle $c \in Z^1(k,J)$ gives
a $\Gamma_k$-set ${}_c\omega$ as basis for $\widehat{{}_cP}$.
Each $\Gamma_k$-set
corresponds to an \'etale algebra ${}_cL$.
Thus $H^2(k,{}_cP) \simeq H^2({}_cL,\bbG_m) \simeq \Br({}_cL)$.

By the injection of~Theorem~\ref{thm:H2main}, we have a composite map
\[ H^1(k,\Aut(X)) \to H^2(k,S \to J) \to H^2(k, P \to J) \]
which is injective if and only if the map  $m_*$ is injective.
Given a fixed cocycle $c \in H^1(k,J)$, we have a morphism
\[ H^1(k,\Aut_1({}_cX)) \to \Br({}_cL) \]
for each neutralization class
(recall the definitiion of $\Aut_1({}_cX)$ from
Remark~\ref{rem:restrictAuto}).

We recall some preliminaries on flasque and coflasque tori
from, for example, \cite{ColSan77La-R-equivalence}.
Let $G$ be a finite group.
A $G$-lattice $M$ is \emph{flasque} (resp. \emph{coflasque}) if
the Tate cohomology group $\widehat{H}^{-1}(H,M)$
(resp. $\widehat{H}^{1}(H,M)$)
is trivial for every subgroup $H \subset G$.
Every $G$-lattice $M$ has a \emph{coflasque resolution}, which is an
exact sequence
\[
1 \to A \to B \to M \to 1
\]
where $A$ is a coflasque $G$-lattice and $B$ is a permutation
$G$-lattice.

For any torus $T$ defined over a field $k$, the action of $\Gamma_k$ on
the character lattice $M$ factors through a finite group $G$.
A torus $T$ is defined to be \emph{flasque}, \emph{coflasque}, or
\emph{invertible} if the corresponding property is true of the
$G$-module $M$.

\begin{lem} \label{lem:injDetails}
Consider an exact sequence of $J$-lattices
\[
\xymatrix{
1 \ar[r] &
\widehat{Q} \ar[r] &
\widehat{P} \ar[r]^{\widehat{m}} &
\widehat{S} \ar[r] &
1 } \]
where $\widehat{Q}$ is coflasque, $\widehat{P}$ is permutation and
$\widehat{S}$ is flasque.
Let $Q$, $P$, and $S$ be the corresponding dual tori
and let $m : S \to P$ the dual morphism obtained from $\widehat{m}$.
Let $m_* : H^2(-, S \to J) \to H^2(-, P \to J)$ be the natural
transformation induced from $m$
as in Definition~\ref{def:H2functoriality}.
Let $m'_* : H^2(-, S) \to H^2(-,P)$ be the usual induced map on abelian
Galois cohomology obtained from $m$.
Then the following are equivalent.
\begin{enumerate}
\item[(a)] $m_* : H^2(-, S \to J) \to H^2(-, P \to J)$ is injective,
\item[(b)] ${}_c(m'_*) : H^2(K, {}_cS ) \to H^2(K, {}_cP)$ is injective
for every field extension $K/k$ and every cocycle $c \in Z^1(K,J)$,
\item[(c)] $H^1(K, {}_cQ)$ is trivial
for every field extension $K/k$ and every cocycle $c \in Z^1(K,J)$,
\item[(d)] $\widehat{Q}$ is invertible,
\item[(e)] $\widehat{S}$ is invertible.
\end{enumerate}
\end{lem}

\begin{proof}
First, we prove the equivalence of (a) and (b).
By the decomposition from Lemma~\ref{lem:altH2desc}, it suffices to
check whether each constituent map
\[ H^2(K, {}_cS)/H^0(K, {}_cJ) \to H^2(k, {}_cP)/H^0(K, {}_cJ) \]
is injective for each $c \in Z^1(K,J)$.
If each ${}_c(m'_*)$ is injective, then so must be each constituent map.
Conversely, since each ${}_c(m'_*)$ is a morphism of groups and $H^0(K, {}_cJ)$
acts by group automorphisms, if the preimage of the trivial element is
trivial then ${}_c(m'_*)$ is injective.

The equivalence of (b) and (c) follow from the triviality of
$H^1(K, {}_cP)$ for all cocycles $c \in Z^1(K,J)$ since
$\widehat{P}$ is a permutation $J$-lattice.

The implication (d) $\implies$ (c) follows since there is a
factorization $\widehat{Q} \to \widehat{M} \to \widehat{Q}$
of the identity morphism for some permutation lattice $M$.
This means that the identity morphism on $H^1(K, {}_cQ)$ factors
through the trivial group $H^1(K, {}_cM)$ for every cocycle $c$.

The implication (c) $\implies$ (d) follows from
Theorem~3.2~of~\cite{Mer10Periods}.
Indeed, we know that, in particular, the generic ${}_cQ$-torsor is
trivial.  The theorem tells us that class of the extension corresponding
to a flasque resolution of $\widehat{Q}$ is trivial.  We conclude that
$\widehat{Q}$ is a direct summand of a permutation module as desired.

Finally, the equivalence of (d) and (e) follows from
Lemma~6~of~\cite{ColSan77La-R-equivalence}.  Indeed, $\widehat{S}$ is
obtained via a flasque resolution of $\widehat{Q}$ and, conversely,
$\widehat{Q}$ is obtained via a coflasque resolution of $\widehat{S}$.
\end{proof}

\begin{prop}[Versality] \label{prop:omegaCoflasque}
Suppose $\widehat{Q}$ is coflasque.
Let $R$ be a torus with a $J$-action such that
the dual $\widehat{R}$ is a permutation $J$-lattice,
and let $q : S \to R$ be a $J$-equivariant morphism of tori.
Then
\[q_* \colon H^2(k,S \to J) \to H^2(k,R \to J)\]
factors through
\[m_* \colon H^2(k,S \to J) \to H^2(k,P \to J)\ . \]
\end{prop}

\begin{proof}
Fix a cocycle $c \in Z^1(k,J)$ and note that it may be viewed as a
morphism $c : \Gamma_k \to J$ since $J$ is a constant group.
The action of $\Gamma_k$ on the character lattices of ${}_cS$, ${}_cR$,
${}_cP$ and ${}_cQ$  then factor through $J$.
We may conclude that ${}_cR$ is quasi-split and
\[1 \to {}_cS \to {}_cP \to {}_cQ \to 1\]
is a coflasque resolution of tori.
Thus, the morphism ${}_c q : {}_cS \to {}_cR$ factors through ${}_c m :
{}_cS \to {}_cP$ by Lemma~4 of~\cite{ColSan77La-R-equivalence}.
We conclude that $q_*$ factors through $m_*$ via
Definition~\ref{def:H2functoriality}.
\end{proof}

We point out the following fact which is essentially due to
Colliot-Th\'el\`ene and Sansuc.

\begin{prop} \label{prop:Sflasque}
The $J$-lattice $\widehat{S}$ is flasque.
\end{prop}

\begin{proof}
We may assume that $k$ is of characteristic $0$
since the statement only depends on data associated to the fan.
Furthermore, by possibly taking a base extension of $k$,
we may assume that there exists a surjection $c : \Gamma_k \to J$.
The twisted torus ${}_cS$ has decomposition group $J$.
Proposition~6~of~\cite{ColSan77La-R-equivalence},
shows that the Picard group $\Pic(\overline{X})$ of a smooth
compactification of a torus is flasque as a $\Gamma_k$-module.
We conclude that ${}_cS$ is flasque as a $\Gamma_k$-module
and thus $\Pic(\overline{X})$ is flasque as a $J$-module.
\end{proof}

\begin{proof}[Proof of Theorem~\ref{thm:inj}]
There exists a coflasque resolution
\[ 1 \to \widehat{Q} \to \widehat{P} \to \widehat{S} \to 1 \]
of $\widehat{S}$. 
By Lemma~\ref{prop:omegaCoflasque}, any morphism as in (a) factors
through the one obtained from this resolution.  Therefore, it suffices
to only consider this morphism.  Since $S$ is flasque by
Proposition~\ref{prop:Sflasque}, we are in the situation of
Lemma~\ref{lem:injDetails}.  Thus (a) and (b) are equivalent.

By Theorem~3.14~of~\cite{Sal84Retract}, a neutral toric variety $X$ is
retract rational if and only if $\Pic(\overline{X})$ is invertible as a
$\Gamma_{k}$ lattice.  Thus (b) and (c) are equivalent.
\end{proof}

\section{Comparison to a construction of Merkurjev-Panin}
\label{sec:MerkPanin}

In \S 7 of \cite{MerPan97K-theory},
A.~Merkurjev and I.~Panin describe a construction which assigns a
separable algebra to each toric variety.  Up to Brauer equivalence,
their construction is essentially equivalent to ours.  In this section,
we make this precise.

In their construction, they fix a smooth projective toric variety $Y$
and an action of a specific torus $T$ on $Y$.
Their notion of isomorphism amounts to $T$-equivariant isomorphisms of
toric varieties.
The set of isomorphism classes is thus in bijection with $H^1(k,T)$ and
corresponds to the class of the open orbit $Y_0$ viewed as a
$T$-torsor.
All the forms of $Y$ belong to the same neutralization class
since $T$ acts trivially on $\Pic(\overline{Y})$.
The neutralization $X$ of $Y$ corresponds to the trivial element of
$H^1(k,T)$ and
there is an injection $H^1(k,T) \to H^1(k,\Aut_1(X))$
where $\Aut_1(X)$ is the group of restricted automorphisms as in
Remark~\ref{rem:restrictAuto}.

First, we review the construction of Section~\ref{sec:H2approx}
in a manner which is more directly comparable.
Note that $\widehat{S}=\Pic(\overline{X})$ is a $\Gamma_k$-lattice.
Choose a morphism $\widehat{m} : \widehat{P} \to \widehat{S}$
of $\Gamma_k$-lattices where $\widehat{P}$ is permutation with basis
$\omega$.
Our choice of $\omega$ gives rise to a quasitrivial torus $P$ and
an \'etale algebra $L$ and we obtain a homomorphism
\[ \alpha_\omega \colon
H^1(k,T) \to H^1(k,\Aut_1(X)) \to \Br(L) \]
which takes the isomorphism class $[Y_0]$ to a Brauer class in
$\Br(L)$.

Now we review the construction of Merkurjev-Panin.
For the moment, assume that $\omega$ is a single $\Gamma_k$-orbit and
thus that $L$ is a field extension of $k$.  The choice of $\omega$
gives us a morphism
\[ S \to P \simeq \Weil{L/k}(\bbG_{m,L}) \]
which is equivalent to a morphism $S_L \to \bbG_{m,L}$,
which, by duality, gives a morphism $\bbZ \to S_L^*$.  This last
morphism can be viewed as an element $\Omega$ of $\Pic(X_L)$.

Fix a Galois splitting field $K/k$ for $T$ which contains $L$.
We obtain an element $[Y_0] \in H^1(\Gamma_{K/k},T(K))$.
From the exact sequence
\[ 1 \to T^* \to \cox{T}^* \to S^* \to 1 \]
and the isomorphism $\Pic(X_L) \simeq H^0(\Gamma_{K/L},S^*)$,
we obtain an element $\partial [\Omega] \in H^1(\Gamma_{K/L},T^*)$
via the connecting homomorphism.
We may take the cup product $[(Y_0)_L] \cup \partial[\Omega]$ in
$H^2(\Gamma_{K/L},K^\times)$ via the standard pairing
$T(K) \otimes T^* \to K^\times$.  Thus, we obtain a map
\[ \beta_\omega \colon H^1(k,T) \to \Br(L) \]
as desired.
When $\omega$ is not a single $\Gamma_k$-orbit, we simply take products
and again obtain an element $\Omega$ in $\Pic(X_L)$ and a morphism
$\beta_\omega$.

\begin{thm} \label{thm:MerkPanin}
For any element $[Y_0] \in H^1(k,T)$,
our construction $\alpha_\omega$ and the Merkurjev-Panin
construction $\beta_\omega$
satisfy the relation
\[\alpha_\omega([Y_0]) = -\beta_\omega([Y_0])\]
in the group $\Br(L)$.
\end{thm}

\begin{proof}
It suffices to prove the theorem for $\omega$ a single $\Gamma_k$-orbit.
First, we claim that the following diagram commutes
\[
\xymatrix{
H^2(k,S) \ar[r]^{m_*} \ar[d] &
H^2(k,P) \ar[d]^{\simeq} \\
H^2(L,S_L) \ar[r]^{\Omega^*} &
H^2(L,\bbG_m)}
\]
where $P=\Weil{L/k}(\bbG_m)$.
Note that the morphisms
\[ m \colon S \to \Weil{L/k}(\bbG_m)\textrm{ and }
D(\Omega) \colon S_L \to \bbG_m \]
arise from the adjunction of the functors $(-)_L$ and $\Weil{L/k}(-)$.
Thus the morphism $m$ factors as
$\Weil{L/k}(D(\Omega)) \circ \eta(S)$ where
$\eta(S) : S \to \Weil{L/k}(S_L)$ is the unit of adjunction.
The morphism $H^2(k,S) \to H^2(L,S_L)$
factors through $H^2(k,S) \to H^2(k,\Weil{L/k}(S_L))$ by the functorial
construction of the restriction map (see \S 2.5~of~\cite{Ser02Galois}).
Thus, the diagram above is the same as 
\[
\xymatrix{
H^2(k,S) \ar[r] &
H^2(k,\Weil{L/k}(S_L)) \ar[r] \ar[d]^{\simeq} &
H^2(k,\Weil{L/k}(\bbG_m)) \ar[d]^{\simeq} \\
&
H^2(L,S_L) \ar[r] &
H^2(L,\bbG_m)
}
\]
which commutes since $\Weil{L/k}(-)(k) \to (-)(L)$ is a natural
transformation of functors from algebraic groups over $k$ to abstract groups.

Recall that $\alpha_\omega([Y_0]) = m_*(\partial[Y_0])$; thus, by the
above, we may instead write
$\alpha_\omega([Y_0]) =
\Omega^*((\partial[Y_0])_L)
= \partial[(Y_0)_L] \cup \Omega$.
Note that $\beta_\omega([Y_0]) = [(Y_0)_L] \cup \partial \Omega$.
Choosing a Galois splitting field $K$ for $L$ with Galois group $G$,
we need to compare two different methods of evaluating the map
\[ H^1(G,T(K)) \otimes H^0(G,S^*) \to H^2(G,K^\times) \ . \]
The theorem follows from Lemma~\ref{lem:homoAlg} below.
\end{proof}

\begin{lem} \label{lem:homoAlg}
Let $G$ be a finite group.  Let $A,B,C$ be $G$-lattices and $I$ be a
$G$-module.  Suppose there is an exact sequence
\[ 1 \to A \to B \to C \to 1 \ . \]
Up to the sign $(-1)^i$, the following square commutes
\[
\xymatrix{
H^i(G,\Hom(A,I)) \otimes H^j(G,C)
\ar[r]^{\partial \otimes \id} \ar[d]^{\id \otimes \partial} &
H^{i+1}(G,\Hom(C,I)) \otimes H^j(G,C)
\ar[d]^{\cup} \\
H^i(G,\Hom(A,I)) \otimes H^{j+1}(G,A)
\ar[r]^{\cup} &
H^{i+j+1}(G,I)
}
\]
for any integers $i \ge 0$, $j \ge 0$.
\end{lem}

\begin{proof}
First, we show that, for any $G$-lattice $M$, the
cup product and evaluation morphism
\[
a: H^i(G,\Hom(M,I)) \otimes H^j(G,M) \to
H^{i+j}(G,I)
\]
is isomorphic to the composition morphism
\[
b: \Ext_G^i(M,I) \otimes \Ext_G^j(\bbZ,M) \to
\Ext_G^{i+j}(\bbZ,I)
 \ .
\]
The isomorphisms between $\Ext$ groups and cohomology groups follow from
Proposition~III.2.2~of~\cite{Bro82Cohomology} since $M$ is a free
$\bbZ$-module.
Picking a free resolution $F_\bullet$ of the $G$-module $\bbZ$,
the maps $a$ and $b$ can be represented by maps of chain complexes
\[
\sHom_G(F_\bullet,\Hom(M,I)) \otimes
\sHom_G(F_\bullet,F_\bullet \otimes M)
\to \sHom_G(F_\bullet,I)
\]
as in V.4.2~of~\cite{Bro82Cohomology}.
The map $a$ takes elements $u \otimes v$ to
$eval_*( (u \otimes \id_M) \circ v)$.
The map $b$ is defined as $\psi(u) \circ v$ where
\[
\psi : \sHom_G(F_\bullet,\Hom(M,I)) \simeq \sHom_G(F_\bullet \otimes M,I)
\]
is the canonical isomorphism.
If $v(f) = \sum_i f_i \otimes m_i$ for some $f \in F_j$ then
we obtain
\[
eval_*( (u \otimes \id_M) \circ v)(f)
= eval_*( \sum_i u(f_i) \otimes m_i )
= \sum_i u(f_i)(m_i) = (\psi(u) \circ v)(f)
\]
for all $u \in \sHom(F_\bullet,\Hom(M,I))$ and
$v \in \sHom(F_\bullet,F_\bullet \otimes M)$
as desired.

Consider arbitrary elements $\alpha \in \Ext^i_G(A,I)$ and
$\gamma \in \Ext^j_G(\bbZ,C)$.
In view of the above, the statement of the lemma is equivalent to
showing that $(\partial \alpha) \circ \gamma = (-1)^i \alpha \circ (\partial
\gamma)$ as elements of $\Ext_G^{i+j+1}(\bbZ,I)$.
The exact sequence is an extension of $C$ by $A$ and thus it may be
considered as an element $\beta$ of $\Ext^1_G(C,A)$.
From Theorem~III.9.1~of~\cite{Mac63Homology} we find that
$\partial(\alpha)=(-1)^i\alpha \circ \beta$ and
$\partial(\gamma) = \beta \circ \gamma$.
Thus, the equality follows from the associativity of the composition
$(\alpha \circ \beta) \circ \gamma = \alpha \circ (\beta \circ \gamma)$
as in Theorem III.5.4~of~\cite{Mac63Homology}.
\end{proof}

\begin{remark}
As an immmediate corollary of Theorem~\ref{thm:MerkPanin}
we see that the Merkurjev-Panin construction is independent of the
choice of torus $T$.  Additionally, their construction can be extended
to arbitrary forms of $X$ rather than a fixed neutralization class by
properly accounting for the group of class group automorphisms $J$.
\end{remark}

\begin{remark}
The construction of Merkurjev-Panin as stated in~\cite{MerPan97K-theory}
actually produces \emph{algebras} rather than simply elements of the
Brauer group.
However, these algebras ultimately depend on a specific choice of torus $T$
and a choice of splitting field.
We discuss a more intrinsic method of producing algebras in the next
section.
\end{remark}

\begin{remark}
Note that $\Omega \in \Pic(X_L)$ is denoted by ``$Q$''
in~\cite{MerPan97K-theory}.
One cannot in general recover $\omega$ from $\Omega$,
as a corresponding set $\omega$ would only be $\Gamma_k$-invariant
rather than $J$-invariant.
However, this is not a serious shortcoming as one can always simply
expand $\omega$ to be $J$-stable.
Indeed, this is necessary if one wants to extend the Merkurjev-Panin
construction to \emph{all} forms of $X$ rather than a single fixed
neutralization class.
\end{remark}

\section{Associated Separable Algebras}
\label{sec:approximation}

In this section, we sharpen the results of Section~\ref{sec:H2approx}
and prove Theorem~\ref{thm:algInj}.
As discussed in Example~\ref{ex:prodProj}, products of projective spaces
are in bijective correspondence with their Cox endomorphism algebras.
Thus, studying the isomorphism classes of separable algebras
is equivalent to studying products of projective spaces.

\begin{thm} \label{thm:orbitToAlgebra}
Let $X$ be a split smooth projective toric variety.
Let $J$ be the group of class group automorphisms of $X$.
There exists a morphism
\[ f : X \to Y := \prod_{i=1}^N \bbP^{n_i} \]
where $n_1, \ldots, n_N$ are positive integers,
and a morphism of functors
\[ f_* : H^1(-,\Aut(X)) \to H^1(-,\Aut_J(Y)) \]
such that, if $X'$ is a form of $X$ and $Y'$ is a form of
$Y$ representing the class $f_*([X'])$, then there is a map
\[ f' \colon X' \to Y' \]
which coincides with $f$ over some field extension.
\end{thm}

\begin{proof}
Given a morphism
\[ f : X \to Y := \prod_{i=1}^N \bbP^{n_i} \]
as above, we obtain an induced map $\Pic(Y) \to \Pic(X)$.
Thus, such morphisms give rise to a
permutation basis $\omega$ of a permutation $J$-lattice
$\widehat{P} := \Pic(Y)$ as in Section~\ref{sec:H2approx}.
Conversely, if every element of $\omega$ corresponds to a sheaf which is
generated by global sections then we obtain such a morphism.

Since $X$ is projective, the ample cone spans $\Pic(\overline{X}) \otimes
\bbR$.  In addition, the fixed locus of $J$ must be a proper closed
subset of $\Pic(\overline{X}) \otimes \bbR$.  Thus, we may assume that
there exists a subset $\omega$ of $\Pic(\overline{X})$ whose corresponding
sheaves are generated by their global sections and such that $J$ acts
faithfully on $\omega$.  This provides the space $Y$ and the desired map
$f : X \to Y$.

The space $Y$ is also a smooth split projective toric variety.
We write $\cox{Y}$ as its characteristic space sitting inside a
vector space $W$ and note that
$\widehat{m} : \widehat{P} \to \Pic(\overline{X})$
is just the morphism $\Pic(Y) \to \Pic(X)$
induced by $f$.

Labeling the basis of $\widehat{P}$ as $\{ E_1, \ldots, E_r \}$
and labelling $\omega = \{ D_1, \ldots, D_r \}$,
we have $\widehat{m}(E_i)=D_i$ for every $i = 1, \ldots, r$.
By duality, we obtain a morphism
$m \colon S \to P$ of tori.

The construction of $f$ is equivalent to an isomorphism
\[
W^\vee = \bigoplus_{i=1}^r H^0(k,\calO_Y(E_i))
\to \bigoplus_{i=1}^r H^0(k,\calO_X(D_i))
\]
which gives rise to a ring homomorphism $F^* : \Cox(Y) \to \Cox(X)$
since $W$ contains all the generators of
$\Cox(Y)$.
The ring homomorphism $F^*$ is equivalent to a morphism $F : V \to W$.
Since each line bundle $\calO_X(D_i)$ is generated by global sections,
the image of $\cox{X}$ is contained in $\cox{Y}$ so we have a restricted
morphism $\cox{f} : \cox{X} \to \cox{Y}$.  This descends to the morphism
$f : X \to Y$ since the map $\cox{f}$ is $S$-equivariant via
$m \colon S \to P$.

Let $B$ be the Cox endomorphism algebra of $Y$.  We have the isomorphism
\[
\GL_1(B) \simeq \prod_{i=1}^r \GL(H^0(k,\calO(E_i))) \ .
\]
Note that $\tAut(Y) \simeq \GL_1(B) \rtimes I$
where $I$ is a finite group containing all possible permutations of
the subgroups in the product.
The group $\GL_1(B) \rtimes J$ is the preimage of the restricted
automorphism group $\Aut_J(Y)$.

The group $\GL_1(A)$ embeds in $\GL_1(B)$ since it has a linear action
on each component of $\Cox(X)_{D_i}$.
The group $J$ embeds in $\GL_1(B) \rtimes J$ by permuting monomials in the
generators of $\Cox(Y)$ since $\omega$ is $J$-stable.
We obtain an $S$-equivariant homomorphism
$\GL_1(A) \rtimes J \to \GL_1(B) \rtimes J$.

This descends to a morphism $\GL_1(A)/S \rtimes J \to \Aut_J(Y)$
after taking the quotient by $S$.
We conclude that the morphism $f$ is
$\GL_1(A)/S \rtimes J$-equivariant.
Using Proposition~\ref{prop:ignoreU},
we obtain $f' : X' \to Y'$ by descent.
\end{proof}

\begin{example}[Blunk's algebras] \label{ex:blunk}
Let $X$ be a split del Pezzo surface of degree $6$.
In the notation of Example~\ref{ex:DP6}, consider the elements
\begin{align*}
&a_1 = H\ ,  \quad
a_2 = 2H - E_1 - E_2 - E_3\ , \\
&b_1 = H - E_1\ , \quad
b_2 = H - E_2\ , \quad
b_3 = H - E_3
\end{align*}
in $\Pic(X)$.  The elements $a_1, a_2$ correspond to morphisms
$X \to \bbP^2$ while $b_1, b_2, b_3$ correspond to morphisms
$X \to \bbP^1$.
Here, $J$ is isomorphic to $S_3 \times S_2$ and
the $J$-orbits are $\omega_1 = \{ a_1, a_2 \}$
and $\omega_2 = \{ b_1, b_2, b_3 \}$.

Let $\omega = \omega_1 \cup \omega_2$.
In this special case, $\Aut_J(Y) = \Aut(Y)$.
Thus Theorem~\ref{thm:orbitToAlgebra} produces a natural transformation
\[ f_* : H^1(-,\Aut(X)) \to H^1(-,\Aut(Y)) \ . \]

Taking $B$ to be the Cox endomorphism algebra of $Y$, we see that
$\omega_1$ produces an Azumaya algebra $B_1$ of rank $3^2$
over an \'etale $k$-algebra of degree $2$;
$\omega_2$, produces an Azumaya algebra $B_2$ of rank $2^2$
over an \'etale $k$-algebra of degree $3$.
Theorem~3.4~of~\cite{Blu10Del-Pezzo} shows $f_*$ is injective
in this case.

(Note that Blunk also specifies embeddings of an \'etale subalgebra
into the algebras $B_1$ and $B_2$.
This is necessary to characterize the image of $f_*$,
but we are only concerned with whether the map is injective.
Blunk's Theorem~3.4 use a stronger notion of
isomorphism on the algebras $B_1$, $B_2$ which respects the subalgebra
embeddings.
However, we show that the stronger notion of isomorphism is equivalent
to the usual one; one can also see this directly.)
\end{example}

As our goal is to show the natural transformation $f_*$ is injective
under suitable conditions, the following example demonstrates why we
want to consider $\Aut_J(Y)$ rather than $\Aut(Y)$ in general.

\begin{example} \label{ex:P1P3toP3P3}
Continuing Example~\ref{ex:P1P3},
consider $X = \bbP^1 \times \bbP^3$ over $\bbR$ and let
\[ f : X \to \bbP^3 \times \bbP^3 \]
be the product of an inclusion $\bbP^1 \to \bbP^3$ as a linear
subspace and the isomorphism $\bbP^3 \to \bbP^3$.
The induced functor
\[
H^1(-,\Aut(X)) \to H^1(-,\Aut(Y))
\]
is \emph{not injective}
since the classes of $C \times \bbP^3$
and $\bbP^1 \times C'$ map to the same element.
However, if we consider the functor
\[
H^1(-,\Aut(X)) \to H^1(-,\Aut_J(Y))
\]
then we do have injectivity.
In this case injectivity can be fixed by a different choice of $\omega$,
but more generally the group $J$ will not be a product of symmetric groups
and such a fix will not exist.
\end{example}

\begin{thm} \label{thm:geomInj}
Suppose we are in the situation of Theorem~\ref{thm:orbitToAlgebra}.
If $\Pic(\overline{X})$ is invertible as a $J$-lattice there is a
canonical choice of maps such that $f_*$ is injective.
\end{thm}

\begin{proof}
The construction of $\omega$ is as follows.
Let $\Nef(X)$ be the Nef cone of $X$.
By Theorem~3.1~of~\cite{Mus02Vanishing}, we see that the divisors in the
Nef cone of $X$ are precisely those generated by their global sections;
moreover, $\Nef(X)$ can be extracted from the fan so it suffices to
assume we are working over $\bbC$.
We may identify $\Nef(X)$
with a strongly convex rational polyhedral cone of full dimension in
$\Pic(X) \otimes \bbR = S^* \otimes \bbR$
(see Theorems~6.3.20 and 6.3.22 of~\cite{CoxLitSch11Toric}).
The intersection $\Nef(X) \cap S^*$ is precisely the set of
line bundles which are generated by their global sections.

For every subgroup $G$ of $J$, consider the intersection
$M_G = \Pic(\overline{X})^G \cap \Nef(X)$.
Since the intersection of a rational polyhedral cone and a subspace
cut out by rational linear equations is again a rational polyhedral cone,
the monoid $M_G$ is finitely generated.
Indeed, since the cones are strongly convex,
each monoid $M_G$ has a canonical minimal generating set $C_G$.
Take $\omega$ to be the union of the $C_G$ for each $G \subset J$.
Note that $\omega$ is $J$-stable since $j(C_G)=C_{j(G)}$ for each $j \in
J$. 

We claim that $C_G$ spans $\Pic(\overline{X})^G$ for every $G \subset J$.
Indeed, we may find a field $K$ and a cocycle $c \in Z^1(K,G)$ such that
$\Pic({}_cX) = \Pic(\overline{X})^G$.  Since projectivity is a geometric
property, the ample cone (and thus the Nef cone) is of full dimension in
$\Pic({}_cX)$.  Thus, $M_G$ is of full dimension in $\Pic(\overline{X})^G$.
Thus $C_G$ spans $\Pic(\overline{X})^G$ as desired.

As in Lemma~3~of~\cite{ColSan77La-R-equivalence},
since we have chosen $\omega$ such that
$\widehat{P}^G \to \Pic(\overline{X})^G$ is
surjective for all subgroups $G \subset J$, we conclude that
the kernel $\widehat{Q}$ is coflasque.
In particular, the action of $J$ is faithful.
Furthermore, every element of $\omega$ corresponds to a sheaf which is
generated by global sections.
The result now follows from Lemma~\ref{lem:injDetails},
Theorem~\ref{thm:orbitToAlgebra} and Lemma~\ref{lem:H1H2} below.
\end{proof}

\begin{lem} \label{lem:H1H2}
In the situation of Theorem~\ref{thm:orbitToAlgebra} we have the
following commutative diagram
\[
\xymatrix{
H^1(k, \Aut(X)) \ar@{^{(}->}[r]^{\delta} \ar[d]^{f_*}&
H^2(k, S \to J) \ar[d]^{m_*} \\
H^1(k, \Aut_J(Y)) \ar@{^{(}->}[r]^{\delta} &
H^2(k, P \to J)
}
\]
where the horizontal maps are as in Theorem~\ref{thm:H2main}.
\end{lem}

\begin{proof}
From the proof of~Theorem~\ref{thm:orbitToAlgebra},
there is a morphism
\[ s : \GL_1(A) \rtimes J \to  \GL_1(B) \rtimes J \]
which descends to the group homomorphism
\[ r : \GL_1(A)/S \rtimes J \to \Aut_J(Y) \]
making $f$ equivariant.

Let $a$ be a cocycle in
$H^1(k, \GL_1(A) \rtimes J) \simeq H^1(k, \Aut(X))$.
To map $a$ to $H^2(k, S \to J)$, we lift to
a continuous function $b : \Gamma_k \to (\GL_1(A) \rtimes J)(k_s)$
and form the $2$-cocycle $\Delta b$ in $Z^2(k,{}_cS)$
for an appropriate cocycle $c \in Z^1(k,J)$.
To map this to $H^2(k, P \to J)$, we simply take the
$2$-cocycle $s(\Delta b)$ in $Z^2(k,{}_cP)$.

Going the other way, map $a$ to $H^1(k, \Aut_J(Y))$ by
the morphism $r$.  To map $r(a)$ to $H^2(k, P \to J)$,
we may take the function $s(b)$ as a lift and
then take $\Delta(s(b))$;
again, this cocycle sits in $Z^2(k,{}_cP)$.

Since $s(\Delta b) = \Delta(s(b))$,
the diagram commutes.
\end{proof}

\begin{remark} \label{rem:coflasqueCanonical}
Note that construction used in the proof of Theorem~\ref{thm:geomInj}
has the advantage of being canonical.  However, it may not be very
economical: in Example~\ref{ex:blunk}, $\omega$ is simply the minimal
generating set for $\Nef(X)$ while our canonical construction produces
a larger set.
\end{remark}

Finally, we are in position to prove Theorem~\ref{thm:algInj}.

\begin{proof}[Proof of Theorem~\ref{thm:algInj}]
By Theorem~\ref{thm:orbitToAlgebra} we may construct a product of projective
spaces $Y = \prod \bbP^{n_i}$ along with a natural transformation
\[
H^1(-,\Aut(X)) \to H^1(-,\Aut_J(Y))
\]
Moreover, since $\Pic(\overline{X})$ is an invertible $J$-lattice,
we may select $Y$ such that the natural transformation is injective
by Theorem~\ref{thm:geomInj}.
Following Example~\ref{ex:prodProj} there is a canonical identification
between the forms of $Y$ and the forms of its Cox endomorphism algebra $B$.
Thus we have a canonical separable algebra $B$ with an injective natural
transformation
\[
H^1(-,\Aut(X)) \to H^1(-,\Aut_J(B))
\]
as desired.
\end{proof}

\section*{Acknowledgements}
The author would like to thank
B.~Antieau,
M.~Borovoi,
M.~Brown,
D.~Cox,
A.~Merkurjev,
Z.~Reichstein,
A.~Ruozzi,
and anonymous referees
for helpful comments and discussions.

\end{document}